\documentclass[11pt,twoside, a4paper]{article}
\usepackage[utf8]{inputenc}
\usepackage[T1]{fontenc}
\usepackage{geometry}
\usepackage{amsmath}
\usepackage{amsfonts}
\usepackage{amssymb}
\usepackage{tikz}
\usepackage{anyfontsize}
\usepackage{graphicx}
\usepackage{rotating}
\usepackage{amsthm}
\usepackage{yfonts}
\usepackage{mathtools}
\usepackage{pdfpages}
\usepackage{float}
\usepackage[title]{appendix}
\usepackage[]{biblatex}
\usepackage{hyperref}
\usepackage{cleveref}
\hypersetup{colorlinks=true, urlcolor=blue}
\usepackage{lastpage}
\usepackage{fancyhdr}
\newcommand{\B}{\mathbb{B}}
\newcommand{\N}{\mathbb{N}}
 
\newcommand{\R}{\mathbb{R}}

\renewcommand{\S}{\mathbb{S}}
\renewcommand{\epsilon}{\varepsilon}

\newtheorem{thm}{Theorem}
\newtheorem{prop}[thm]{Proposition}
\newtheorem{corollaire}[thm]{Corollary}
\newtheorem{lemma}[thm]{Lemma}
\theoremstyle{definition}

\newtheorem{defn}[thm]{Definition}
\newtheorem{rem}[thm]{Remark}

\newtheorem{question}[thm]{Question}

\title {Sharp upper bounds for Steklov eigenvalues of a hypersurface of revolution with two boundary components in Euclidean space}

\author{Léonard Tschanz}

\author{Léonard Tschanz 
\vspace{1cm} \\
Institut de Mathématiques, Neuchâtel University, Neuchâtel, Switzerland \\
leonard.tschanz@unine.ch \\
ORCID: 0000-0003-2097-6759}

\date{}

\begin{document}

\maketitle

\begin{abstract}
    We investigate the question of sharp upper bounds for the Steklov eigenvalues of a hypersurface of revolution in Euclidean space with two boundary components, each isometric to $\S^{n-1}$. 
    For the case of the first non zero Steklov eigenvalue, we give a sharp upper bound $B_n(L)$ (that depends only on the dimension $n \ge 3$ and the meridian length $L>0$) which is reached by a degenerated metric $g^*$ that we compute explicitly. We also give a sharp upper bound $B_n$ which depends only on $n$. 
    Our method also permits us to prove some stability properties of these upper bounds. 
    \medskip

    \noindent\textsc{R\'esum\'e.}
    Nous étudions la question des bornes supérieures optimales pour les valeurs propres de Steklov d'une hypersurface de révolution de l'espace euclidien avec deux composantes connexes du bord, chacune isométrique à $\S^{n-1}$. 
    Dans le cas de la première valeur propre de Steklov non nulle, nous donnons une borne supérieure optimale $B_n(L)$ (qui ne dépend que de la dimension $n$ et de la longueur d'un méridien $L >0$) qui est atteinte par une métrique dégénérée $g^*$ que l'on calcule explicitement. Nous donnons aussi une borne supérieure optimale $B_n$ qui ne dépend que de $n$. 
    Notre méthode nous permet également de prouver des propriétés de stabilité que possèdent ces bornes supérieures.
    \medskip
    
    \noindent \textbf{Keywords:} Spectral geometry, Steklov problem, hypersurfaces of revolution, sharp upper bounds.
    \medskip

\end{abstract}

\section{Introduction}

Let $(M, g)$ be a smooth compact connected Riemannian manifold  of dimension $n \ge 2$ with smooth boundary $\Sigma$. The Steklov problem on $(M, g)$ consists of finding the real numbers $\sigma$ and the harmonic functions $f : M \longrightarrow \R$ such that $\partial_\nu f=\sigma f$ on $\Sigma$, where $\nu$ denotes the outward normal on $\Sigma$. Such a $\sigma$ is called a Steklov eigenvalue of $(M, g)$. It is well known that the Steklov spectrum forms a discrete sequence $0= \sigma_0(M,g) < \sigma_1(M,g) \le \sigma_2(M,g) \le \ldots \nearrow \infty$. Each eigenvalue is repeated with its multiplicity, which is finite. If the context is clear, then we simply write $\sigma_k(M)$ for $\sigma_k(M,g)$.
\medskip

It is known \parencite[Thm. 1.1]{CEG2} that for any connected compact manifold $(M, g)$ of dimension $n \ge 3$, there exists a family $(g_\epsilon)$ of Riemannian metrics conformal to $g$ which coincide with $g$ on the boundary of $M$, such that
\begin{align*}
    \sigma_1(M, g_\epsilon) \underset{\epsilon \to 0}{\longrightarrow} \infty.
\end{align*}
Therefore, to obtain upper bounds for the Steklov eigenvalues, it is necessary  to study manifolds that satisfy certain additional constraints. We refer to \cite{colbois2022recent} for an overview of the current state-of-the-art on geometric upper bounds for the Steklov eigenvalues.
\medskip

Recently, authors investigated the Steklov problem on manifolds of revolution \cite{FTY, FS, XC, XC2}.  A natural constraint for the manifolds is that they are (hyper)surfaces of revolution in Euclidean space. Some work has already been done on these kinds of manifolds, see for example \cite{CGG, CV}. We refer to \parencite[Sect. 3.1]{CGG} for a review about what these manifolds are, and consider a particular case in this paper that we define below (see \Cref{defn : revolution}).
\medskip

This work led to the discovery of lower and upper bounds for the Steklov eigenvalues of a hypersurface of revolution. We begin by recalling some recent results. 
\medskip

We first consider results for hypersurfaces of revolution with one boundary component that is isometric to $\S^{n-1}$. In dimension $n=2$, it is proved in \parencite[Prop. 1.10]{CGG}  that each surface of revolution $M \subset \R^3$ with boundary $\S^1 \subset \R^2 \times \{0\}$ is Steklov isospectral to the unit disk. In dimension $n \ge 3$, many bounds were given. It is proved that each hypersurface of revolution $M \subset \R^{n+1}$ with one boundary component isometric to $\S^{n-1}$  satisfies $\sigma_k(M) \ge \sigma_k(\B^n)$, where $\B^n$ is the Euclidean ball and equality  holds if and only if $M= \B^n \times \{0\}$, see \parencite[Thm. 1.8]{CGG}.
In \parencite[Thm. 1]{CV}, the authors show the following upper bound: if $M \subset \R^{n+1}$ is a hypersurface of revolution with one boundary component isometric to $\S^{n-1}$, then for each $k \ge 1$, we have 
\begin{align*}
    \sigma_{(k)}(M) < k+n-2,
\end{align*}
where $\sigma_{(k)}(M)$ is the $k$th distinct Steklov eigenvalue of $M$. Although there exists no equality case within the collection of hypersurfaces of revolution, this upper bound is sharp. Indeed, for each $\epsilon >0$ and each $k\ge 1$, there exists a hypersurface of revolution $M_\epsilon$ such that $\sigma_{(k)}(M_\epsilon) > k+n-2-\epsilon$.
\medskip

These results concern hypersurfaces of revolution that have one boundary component isometric to $\S^{n-1}$. Therefore, the goal of this paper is to investigate the Steklov problem on a hypersurface of revolution with two boundary components. As was already done in \cite{CGG} and in \cite{CV}, we will consider hypersurfaces with boundary components isometric to  $\S^{n-1}$. We begin by defining the context.

\begin{defn} \label{defn : revolution}
An $n$-dimensional compact hypersurface of revolution $(M,g)$ in Euclidean space  with two boundary components each isometric to  $\S^{n-1}$ is the warped product $M=[0, L] \times \S^{n-1}$ endowed with the Riemannian metric 
\begin{align*}
    g(r,p) = dr^2 + h^2(r)g_0(p),
\end{align*}
where $(r,p) \in [0, L] \times \S^{n-1}$, $g_0$ is the canonical metric of the $(n-1)$-sphere of radius one and $h : [0, L] \longrightarrow \R_+^*$ is a smooth function which  satisfies:
\begin{enumerate}
    \item[(1)] $|h'(r)| \le 1$ for all $r \in [0, L]$;
    \item[(2)] $h(0)=h(L)=1$.
\end{enumerate}

Assumption $(1)$ comes from the fact that $(M, g)$ is a hypersurface in Euclidean space $\R^{n+1}$, see \parencite[Sect. 3.1]{CGG} for more details. Assumption $(2)$ implies that each component of the boundary is isometric to  $\S^{n-1}$, as commented in Fig. \ref{fig : surface  revolution def}.
\medskip

We now make some remarks on the terminology used throughout this paper.
If $M = [0, L] \times \S^{n-1}$ and $h : [0, L] \longrightarrow \R_+^*$ satisfies the properties above, we say  that $M$ is a \textit{hypersurface of revolution}, we say that  $g(r, p) = dr^2 + h^2(r)g_0(p)$ is a \textit{metric of revolution on $M$ induced by $h$} and we call the number $L$ the \textit{meridian length} of $M$.

\end{defn}

\begin{figure}[H]
\centering
\includegraphics[scale=0.6]{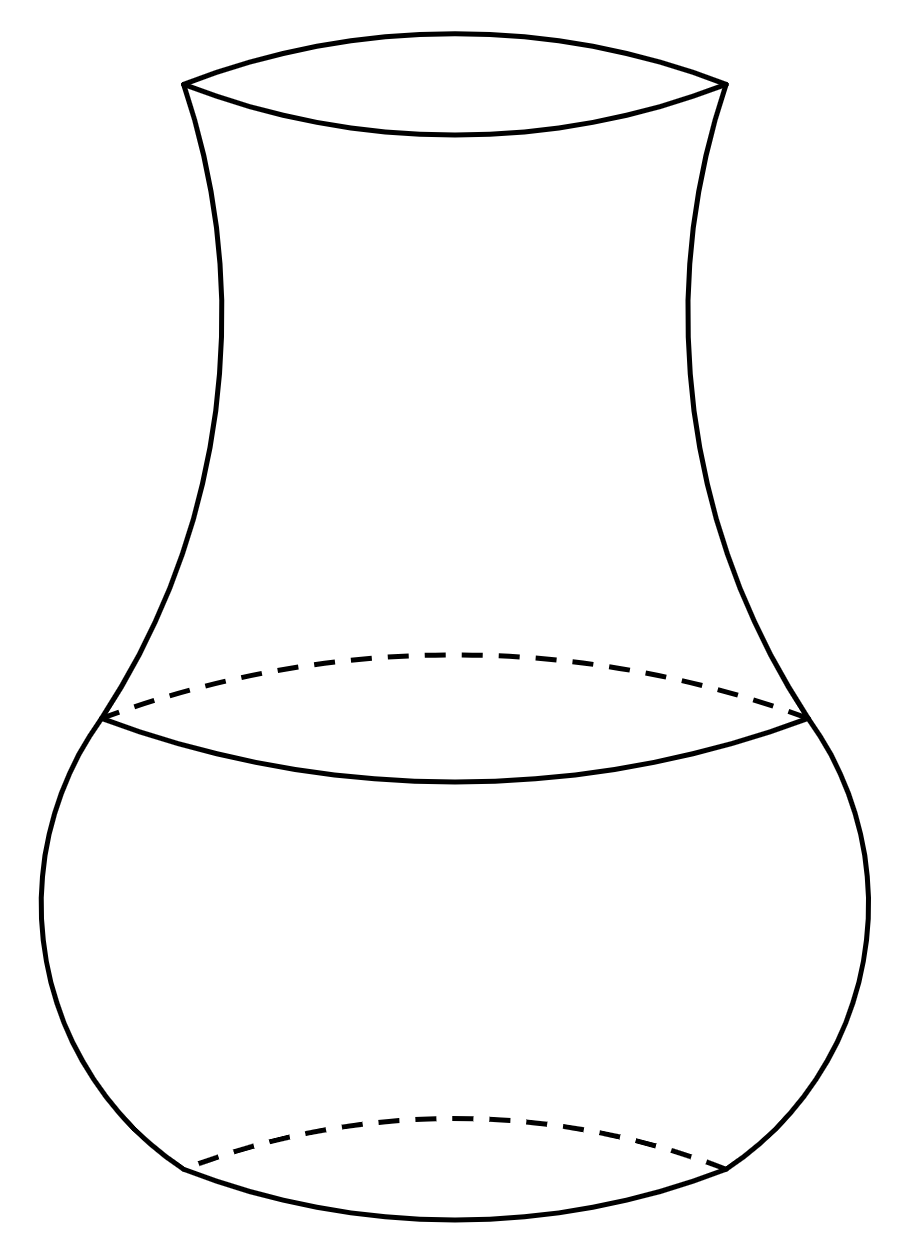}
\caption{Since $h(0)=h(L)=1$, the boundary of $M$ consists of two copies of $\S^{n-1}$.}
\label{fig : surface revolution def}
\end{figure}

Some lower bounds have already been obtained is this case. Indeed, \parencite[Thm. 1.11]{CGG} states that if  $M \subset \R^{n+1}$, $n \ge 3$, is a hypersurface of revolution (in the sense of \Cref{defn : revolution}), and $L>2$ is the meridian length of $M$, then for each $k \ge 1$, 
\begin{align*}
    \sigma_k(M) \ge \sigma_k(\B^n \sqcup \B^n).
\end{align*}
Moreover, this inequality is sharp. In the case $0 < L \le 2$, a lower bound is also obtained:
\begin{align*}
    \sigma_k(M) \ge \left( 1- \frac{L}{2}\right)^{n-1} \sigma_k(C_L, dr^2 + g_0),
\end{align*}
 However, this inequality does not appear to be sharp.
\medskip

In this paper, we will look for \textit{upper bounds} for the Steklov eigenvalues of hypersurfaces of revolution. First, we recall that there exists a bound $B_n^k(L)$ such that for all metrics of revolution $g$ on $M$, we have $\sigma_k(M, g) < B_n^k(L)$. Indeed, Proposition 3.3 of \cite{CGG} states that if $M = [0, L] \times \S^{n-1}$ is a hypersurface of revolution, then we have
\begin{align*}
    \sigma_k(M) \le \left(1+ \frac{L}{2}\right)^{n-1} \sigma_k(C_L, dr^2 + g_0).
\end{align*}
As such, a natural question is the following:
\begin{center}
    \textit{Given the dimension $n \ge 3$ and the meridian length $L$ of $M$, does a metric of revolution $g^*$ on $M$ exist, such that $\sigma_k(M,g) \le \sigma_k(M,g^*)$ for all metrics of revolution $g$ on $M$? } 
\end{center}

Our investigations show that the answer is negative. Indeed, a sharp upper bound $B_n^k(L)$ exists, but no metric of revolution on $M=[0,L]\times \S^{n-1}$ achieves the equality case. However, there exists a non-smooth metric $g^*$, that we will call a \textit{degenerated maximizing  metric}, which maximizes the $k$th Steklov eigenvalue, for each $k \in \N$. This metric is non-smooth, therefore $g^*$ is not a metric of revolution on $M$ in the sense of Definition \ref{defn : revolution}. 
Endowed with this metric, $(M,g^*)$ can be seen as two annuli glued together; we provide more information about this degenerated maximizing  metric $g^*$  and the geometric representation of $(M, g^*)$ in Sect. \ref{sect : proof}.
\medskip

We state our first result:
\begin{thm} \label{thm : pricipal}
Let $(M=[0,L] \times \S^{n-1}, g_1)$ be a hypersurface of revolution in Euclidean space with two boundary components each isometric to   $\S^{n-1}$  and meridian length~$L$. We suppose $n \ge 3$. Then there exists a metric of revolution $g_2$ on $M$ such that for each $k \ge 1$,
\begin{align*}
    \sigma_k(M,g_1) < \sigma_k(M, g_2).
\end{align*}
\end{thm}
This result implies that among all metrics of revolution on $M$, none maximizes the $k$th non zero Steklov eigenvalue. 
Nevertheless, given any metric of revolution $g_1$ on $M$, we can iterate Theorem \ref{thm : pricipal} to generate a sequence of metrics $(g_i)_{i=1}^\infty$ on $M$. This sequence converges to a unique non-smooth metric   $g^*$ on $M$, which is quite simple (see Sect. \ref{sect : proof}) and which maximizes the $k$th Steklov eigenvalue. That is why we call $g^*$ the degenerated maximizing metric. Hence, as we search for the optimal bounds $B_n^k(L)$, we must use information contained in  $g^*$.
\medskip

We start by studying the case $k=1$. We fix $n \ge 3$ and $L >0$ and search for a sharp upper bound $B_n(L)$ for $\sigma_1(M,g)$. In this case, we are able to calculate an expression for $B_n(L)$:
\begin{thm}  \label{thm : principal deux}
Let $(M= [0, L] \times \S^{n-1}, g)$ be a hypersurface of revolution  in Euclidean space with two boundary components each isometric to  $\S^{n-1}$ and dimension $n \ge 3$. Then the first non trivial Steklov eigenvalue $\sigma_1(M,g)$ is bounded above, by a bound that depends only on the dimension $n$ and the meridian length $L$ of $M$:
\begin{align*}
    \sigma_1(M,g) < B_n(L) := \min\left\{ \frac{(n-2)\left( 1+L/2 \right)^{n-2}}{\left( 1+L/2 \right)^{n-2}-1}, \frac{(n-1)\left( \left(1+L/2\right)^n -1 \right)}{\left( 1+ L/2 \right)^n +n-1}\right\}.
\end{align*}
Moreover, this bound is sharp: for each $\epsilon >0$, there exists a metric of revolution $g_\epsilon$ on $M$ such that $\sigma_1(M, g_\epsilon) >B_n(L)- \epsilon$.
\end{thm}

We have the following asymptotic behaviour:
\begin{align*}
    B_n(L) & \underset{L \to \infty}{\longrightarrow} n-2 \\
    B_n(L)  & \underset{L \to 0}{\longrightarrow } 0,
\end{align*}
see Fig. \ref{fig : borne}.
\medskip

We also study the function $L \longmapsto B_n(L)$. This allows us to find a sharp upper bound $B_n$ such that for all meridian lengths $L>0$ and  metrics of revolution $g$ on $M$, we have $\sigma_1(M,g) < B_n$:

\begin{corollaire} \label{cor : corollaire du principal deux}
Let $n \ge 3$. 
 Then there exists a bound $B_n < \infty$ such that for all hypersurfaces of revolution $(M, g)$  in Euclidean space with two boundary components each isometric to  $\S^{n-1}$, we have
\begin{align*}
    \sigma_1(M, g) < B_n := \frac{(n-2)\left(1+\frac{L_1}{2}\right)^{n-2}}{\left(1+\frac{L_1}{2}\right)^{n-2}-1},
\end{align*}  
    where $L_1$ is the unique real positive solution of the equation
\begin{align*}
    \left(1+L/2\right)^{2n-2}-(n-1)\left(1+L/2\right)^n-(n-1)^2\left(1+L/2\right)^{n-2}+n-1=0.
\end{align*}
Moreover, this bound is sharp: for each $\epsilon >0$, there exists a hypersurface of revolution  with two boundary components each isometric to a unit sphere $(M_\epsilon, g_\epsilon)$  such that $\sigma_1(M_\epsilon, g_\epsilon) > B_n - \epsilon$.
\end{corollaire}

We say that $L_1$ is a \textit{critical length} associated with $k=1$, see Definition \ref{def : critical length}.

\begin{prop} \label{rem : dimension sur L1}
Let $n \ge 3$, and let $L_1 = L_1(n)$ be the critical length associated with $k =1$. Then we have: 
\begin{align*}
    \lim_{n\to \infty} L_1(n) = 0
\; \mbox{ and } \;
     \lim_{n\to \infty} B_n = \infty.
\end{align*}
\end{prop}
Note that the behaviour of $L_1$ is surprising since we know  that when $n$ is fixed, then $L \ll 1$ implies  $\sigma_1(M,g) \ll 1$. Indeed, by \parencite[Prop. 3.3]{CGG}, we have
\begin{align*}
    \sigma_1(M) \le \left(1+ \frac{L}{2} \right)^{n-1} \sigma_1(C_L) \underset{L \to 0}{\longrightarrow} 0.
\end{align*}

Now that we have provided information about sharp upper bounds for $\sigma_1(M, g)$, it is natural to wonder what kind of stability properties the hypersurfaces of revolution possess. A first interesting question is the following:
\begin{center}
    \textit{Given the information that $\sigma_1(M=[0, L] \times \S^{n-1}, g)$ is close to the sharp upper bound $B_n$, can we conclude that the meridian length $L$ of $M$ is close to the critical length $L_1$?}
\end{center}
The answer to this question is positive. Indeed we will prove that if $L$ is not close to $L_1$, then $\sigma_1(M, g)$ is not close to $B_n$. Additionally, given the information that $\sigma_1(M, g)$ is $\delta$-close to $B_n$, we will show that the distance between $L$ and $L_1$ is less than $\delta$, up to a constant of proportionality which depends only on the dimension $n$.

\begin{thm} \label{thm : stability 1}
Let $M=[0, L] \times \S^{n-1}$,  with $L >0$ and $n \ge 3$. We suppose $L \ne L_1$. Then there exists a constant $C(n, L) >0$ such that for all metrics of revolution $g$ on $M$, we have 
\begin{align*}
    B_n - \sigma_1(M, g) \ge C(n, L).
\end{align*}
Moreover, there exists a constant $C(n) > 0$ such that for all $0 < \delta < \frac{B_n - (n-2)}{2}$, we have 
\begin{align*}
    |B_n - \sigma_1(M, g) | < \delta \implies |L_1 -L| < C(n) \cdot \delta.
\end{align*}
\end{thm}

We also consider the following question about stability properties:
\begin{center}
    \textit{Given the information that $\sigma_1(M, g)$ is close to the sharp upper bound $B_n(L)$, can we conclude that the metric of revolution $g$ is close (in a sense that is defined below) to the degenerated maximizing metric $g^*$?}
\end{center}
We prove that if $g$ is not close to $g^*$, then $\sigma_1(M,g)$ is not close to $B_n(L)$.
\medskip

For this purpose, given  $m \in [1, 1+L/2)$, we define 
\begin{align*}
    \mathcal{M}_m :=  & \{\mbox{metrics of revolution } g \mbox{ on } M \\
       & \mbox{ induced by a function } h  \mbox{ such that } \max_{r \in [0, L]} \{h(r)\} \le m \}.
\end{align*}
The collection $\mathcal{M}_m$ can be thought of the set of all metrics of revolution that are not close to the degenerated maximizing metric $g^*$, where the qualitative appreciation of the word "close" is given by the parameter $m$. The larger $m$ is, the closer to $g^*$ the metrics in $\mathcal{M}_m$ can be.
\medskip

We get the following result:

\begin{thm} \label{thm : stability k}
Let $(M=[0, L] \times \S^{n-1}, g)$ be a hypersurface of revolution in Euclidean space with two boundary components each isometric to  $\S^{n-1}$ and dimension $n \ge 3$. Let $m \in [1, 1+L/2)$ and $\mathcal{M}_m$ as above. Then there exists a constant $C(n, L, m) >0$ such that for all $g \in \mathcal{M}_m$, we have 
\begin{align*}
    B_n(L)- \sigma_1(M, g) \ge C(n, L, m).
\end{align*}
\end{thm}



These results solve the case $k=1$.  Therefore, it would be interesting to find the same kind of results for any $k \ge 1$. After having calculated sharp upper bounds for some higher values of $k$ in Sect. \ref{subsect : sigma2 sigmam1} and \ref{subsect : sigmam1plus1}, we will see that in order to get an expression for $B_n^k(L)$, we need to distinguish between many cases. As such, giving a general formula for $B_n^k(L)$ or $B_n^k := \sup_{L \in \R_+^*} \{B_n^k(L)\}$ via this method seems difficult. We discuss this in \Cref{rem : borne sans L}.

\begin{defn} \label{def : critical length}
We say that $L_k \in \R_+^*$ is a finite critical length associated with $k$ if we have $B_n^k = B_n^k(L_k)$. We say that $k$ has a critical length at infinity if it satisfies $B_n^k= \lim_{L\to \infty} B_n^k(L)$.
\end{defn}

These lengths are critical in the following sense: if $L_k \in \R_+^*$ is a finite critical length for a certain $k \in \N$ and if we write $g^*$ the degenerated maximizing metric on $M_k=[0, L_k]\times \S^{n-1}$, then 
\begin{align*}
    B_n^k= \sigma_k(M_k,g^*).
\end{align*}

Given $n \ge 3$, there exist some $k$ which have  a finite critical length associated with them. Indeed, thanks to \Cref{cor : corollaire du principal deux}, we know that $k=1$ has this property. Moreover, we know that there exist some $k$ which have a critical length at infinity, see Sect. \ref{subsect : sigma2 sigmam1}.
\medskip

 Since we want to study upper bounds for the Steklov eigenvalues, it is then natural to ask what qualitative and quantitative information  we can provide about these critical lengths.
\medskip

We get the following result:
\begin{thm} \label{thm : principal 4}
Let $n \ge 3$. Then there exist infinitely many $k \in \N$ which have a finite critical length associated with them. 
Moreover, if we call  $(k_i)_{i=1}^\infty \subset \N$ the increasing sequence of such $k$ and if we call $(L_i)_{i=1}^\infty$ the associated sequence of finite critical lengths, then we have 
\begin{align*}
    \lim_{i\to \infty }L_i=0.
\end{align*}
\end{thm}

The existence of finite critical lengths is something surprising when we compare with what happens in the case of hypersurfaces of revolution with one boundary component. Indeed, using our vocabulary, we can state that in the case of hypersurfaces of revolution with one boundary component, each $k \in \N$ has a critical length at infinity, see \parencite[Prop. 7]{CV}. Nevertheless, in our case, Theorem \ref{thm : principal 4} guarantees that there exist infinitely many $k \in \N$ which have a finite critical length associated with them. Moreover,  we will show in  Sect. \ref{subsect : sigma2 sigmam1} that there exist some $k$ which have a critical length at infinity. However, we do not know if there are \textit{infinitely many} of them. This consideration leads to the following open question (Question  \ref{q : open}):
\begin{center}
    \textit{Given $n \ge 3$, are there finitely or infinitely many $k \in \N$ such that $k$ has a critical length at infinity?} 
\end{center}

\textbf{Plan of the paper.}
In Sect. \ref{sect : var car}, we recall  the variational characterizations of the Steklov eigenvalues before giving the expression of eigenfunctions on hypersurfaces of revolution, and we introduce the notion of mixed Steklov-Dirichlet and Steklov-Neumann problems and state some propositions about them. We will then have enough information to prove Theorem \ref{thm : pricipal} in Sect. \ref{sect : proof}.  This will allow us to prove Theorem \ref{thm : principal deux}, Corollary \ref{cor : corollaire du principal deux} and \Cref{rem : dimension sur L1} in Sect. \ref{sect : proof 2}. Then we prove the stability properties of hypersurfaces of revolution, i.e Theorem \ref{thm : stability 1} and Theorem \ref{thm : stability k} in Sect. \ref{sect : stability}. We continue by performing some calculation for sharp upper bounds for higher eigenvalues in Sect. \ref{sect : 6}.  We conclude by  proving Theorem \ref{thm : principal 4} in Sect. \ref{sect : proof infinity de k avec L fini}. 
\medskip

\textbf{Acknowledgment.} I would like to warmly thank my thesis supervisor Bruno Colbois for offering me the opportunity to work on this topic, and for his precious help which enabled me to solve the difficulties encountered. I also want to thank several of  my friends and colleagues, Antoine Bourquin, Laura Grave De Peralta and Flavio Salizzoni,  for various discussions we had, their help and advice. Moreover, I would like to thank the anonymous referee for their careful proofreading and their comments which have improved the paper greatly.

\section{Variational characterization of the Steklov eigenvalues and mixed problems} \label{sect : var car}

We state some general facts about Steklov eigenfunctions and define the mixed Steklov-Dirichlet and Steklov-Neumann problems.

\subsection{Variational characterization of the Steklov eigenvalues} 

Let $(M, g)$ be a Riemannian manifold with smooth boundary $\Sigma$. Then we can characterize the $k$th Steklov eigenvalue of $M$ by the following formula:
\begin{align} \label{form : car var}
    \sigma_k(M, g) = \min \left\lbrace R_g(f)  : f \in H^1(M), \; f \perp_\Sigma f_0, f_1, \ldots, f_{k-1} \right\rbrace,
\end{align}
where 
\begin{align*}
    R_g(f) = \frac{\int_M |\nabla f|^2 dV_g}{\int_\Sigma |f|^2 dV_\Sigma} 
\end{align*}
 is called the Rayleigh quotient  and
\begin{align*}
    f \perp_\Sigma f_i \iff \int_\Sigma f f_i dV_\Sigma =0.
\end{align*}


Another way to characterize the $k$th eigenvalue of $M$ is given by the Min-Max principle:
\begin{align} \label{form : min max}
    \sigma_k(M,g) = \min_{E \in \mathcal{H}_{k+1}(M)} \max_{0 \neq f \in E} R_g(f),
\end{align}
where $\mathcal{H}_{k+1}$ is the set of all $(k+1)$-dimensional subspaces in the Sobolev space $H^1(M)$.

\medskip

We state now a proposition that provides us with information about the expression of the Steklov eigenfunctions of a hypersurface of revolution.  
\medskip

We denote by $0 =\lambda_0 < \lambda_1 \le \lambda_2 \le \ldots \nearrow \infty$ the spectrum of the Laplacian on $(\S^{n-1}, g_0)$ and we consider $(S_j)_{j=0}^\infty$ an orthonormal basis of eigenfunctions associated to $(\lambda_j)_{j=0}^\infty$. 
\begin{prop} \label{prop : sep variables}
Let $(M, g)$ be a hypersurface of revolution  as in \Cref{defn : revolution}. Then each eigenfunction on $M$ can  be written as $f_k(r, p) = u_l(r)S_j(p)$, where  $u_l$ is a smooth function on $[0, L]$.
\end{prop}

This property is well known for warped product manifolds (and thus for our case of hypersurfaces of revolution) and it is used often, see for example  \parencite[Remark 1.1]{DHN},  \parencite[Lemma 3]{Esc}, \parencite[Prop. $3.16$]{T2} or  \parencite[Prop. $9$]{XC}.




\subsection{Mixed problems and their variational characterizations} \label{sect : mixed problems}



Let $(N, \partial N)$ be a smooth compact connected Riemannian manifold and $A \subset N$ be a domain which satisfies $\partial N \subset \partial A$. We suppose that $\partial A$ is smooth and we call $\partial_{int}A$ the intersection of $\partial A$ with the interior of $N$.
\begin{defn}
The Steklov-Dirichlet problem on $A$ is the eigenvalue problem 
\begin{align*}
 \left\{
    \begin{array}{ll}
        \Delta f  =0& \mbox{ in } A  \\
        \partial_\nu f = \sigma f & \mbox{ on } \partial N \\
        f  =0 & \mbox{ on } \partial_{int}A.
    \end{array}
\right.
\end{align*}

\end{defn}

It is well known that this mixed problem possesses solutions that form a discrete sequence
\begin{align*}
    0 < \sigma_0^D(A) \le \sigma_1^D(A) \le \ldots \nearrow \infty.
\end{align*}

The variational characterization of the $k$th Steklov-Dirichlet eigenvalue is the following:
\begin{align*}
    \sigma_k^D(A) = \min_{E \in \mathcal{H}_{k+1,0}(A)} \max_{0 \neq f \in E} \frac{\int_A |\nabla f|^2 dV_A}{\int_\Sigma |f|^2 dV_\Sigma},
\end{align*}
where $\mathcal{H}_{k+1, 0}$ is the set of all $(k+1)$-dimensional subspaces in the Sobolev space 
\begin{align*}
    H_0^1(A)= \{f \in H^1(A) : f=0 \mbox{ on } \partial_{int}A \}.
\end{align*}

\begin{defn}
The Steklov-Neumann problem on $A$ is the eigenvalue problem 
\begin{align*}
 \left\{
    \begin{array}{ll}
        \Delta f  =0& \mbox{ in } A  \\
        \partial_\nu f = \sigma f & \mbox{ on } \partial N \\
        \partial_\nu f = 0 & \mbox{ on } \partial_{int}A.
    \end{array}
\right.
\end{align*}

\end{defn}

It is well known that this mixed problem possesses solutions that form a discrete sequence
\begin{align*}
    0 = \sigma_0^N(A) \le \sigma_1^N(A) \le \ldots \nearrow \infty.
\end{align*}

The variational characterization of the $k$th Steklov-Neumann eigenvalue is the following:
\begin{align*}
    \sigma_k^N(A) = \min_{E \in \mathcal{H}_{k+1}(A)} \max_{0 \neq f \in E} \frac{\int_A |\nabla f|^2 dV_A}{\int_\Sigma |f|^2 dV_\Sigma},
\end{align*}
where $\mathcal{H}_{k+1}$ is the set of all $(k+1)$-dimensional subspaces in the Sobolev space $H^1(A)$.

\subsection{Mixed problems on annular domains}

Let $\B_1$ and $\B_R$ be the balls in $\R^n$, $n \ge 3$,  with radius $1$ and $R >1$ respectively centered at the origin. The annulus $A_R$ is defined as follows: $A_R = \B_R \backslash \overline{\B}_1$. We say that this annulus is of inner radius $1$ and outer radius $R$. This particular kind of domain  shall be useful in this paper.
\medskip

For such domains, it is possible to compute $\sigma_{(k)}^D(A_R)$ explicitly, which is the $(k)$th eigenvalue of the Steklov-Dirichlet problem on $A_R$, counted without multiplicity.
\medskip

We state here Proposition $4$ of \cite{CV}:
\begin{prop} \label{prop : SD}
For $A_R$ as above, consider the Steklov-Dirichlet problem
\begin{align*}
    \left\{
    \begin{array}{ll}
       \Delta f = 0  & \mbox{ in } A_R  \\
        \partial_\nu f = \sigma f & \mbox{ on } \partial \B_1 \\
        f = 0 & \mbox{ on } \partial \B_R.
    \end{array}
    \right.
\end{align*}
Then, for $k \ge 0$, the $(k)$th eigenvalue (counted without multiplicity) of this problem is 
\begin{align*}
    \sigma_{(k)}^D(A_R) = \frac{(k+n-2)R^{2k+n-2}+k}{R^{2k+n-2}-1}.
\end{align*}
\end{prop}


By \parencite[Prop. 4]{CV}, it is possible to get the expression of the eigenfunctions of the Steklov-Dirichlet problem on an annular domain.

\begin{lemma} \label{lem : Dirichlet}
Each eigenfunction $\varphi_l$ of the Steklov-Dirichlet problem on the annulus $A_{R}$ can be expressed as $\varphi_l(r,p)=\alpha_l(r)S_l(p)$, where $S_l$ is an eigenfunction for the $l^{th}$ harmonic of the sphere $\mathbb{S}^{n-1}$.
\end{lemma}

It is possible to compute $\sigma_{(k)}^N(A_R)$ explicitly, which is the $(k)$th eigenvalue of the Steklov-Neumann problem on $A_R$, counted without multiplicity.
\medskip

We state now Proposition $5$ of \cite{CV}:
\begin{prop} \label{prop : SN}
For $A_R$ as above, consider the Steklov-Neumann problem
\begin{align*}
    \left\{
    \begin{array}{ll}
       \Delta f = 0  & \mbox{ in } A_R  \\
        \partial_\nu f = \sigma f & \mbox{ on } \partial \B_1 \\
        \partial_\nu f = 0 & \mbox{ on } \partial \B_R.
    \end{array}
    \right.
\end{align*}
Then, for $k \ge 0$, the $(k)$th eigenvalue (counted without multiplicity) of this problem is 
\begin{align*}
    \sigma_{(k)}^N(A_R) = k \frac{(k+n-2)(R^{2k+n-2}-1)}{kR^{2k+n-2}+k+n-2}.
\end{align*}
\end{prop}

In the same manner as before, we have the following expression for the Steklov-Neumann eigenvalues, see \parencite[Prop. 5]{CV}.

\begin{lemma} \label{lem : Neumann}
Each eigenfunction $\phi_l$ of the Steklov-Neumann problem on the annulus $A_{R}$ can be expressed as $\phi_l(r,p)=\beta_l(r)S_l(p)$, where $S_l$ is an eigenfunction for the $l^{th}$ harmonic of the sphere $\mathbb{S}^{n-1}$.
\end{lemma}

\section{The degenerated maximizing metric}  \label{sect : proof}


A particular case of hypersurfaces of revolution 
is the following: let $M=[0, L] \times \S^{n-1}$ be endowed with a metric of revolution $g(r,p)=dr^2+h^2(r)g_0(p)$. 
Let us suppose that there exists $\epsilon >0$ such that $h(r) = 1+r$ on $[0, \epsilon]$. 
Let us consider the connected component of the boundary $\S_0$ associated with $h(0)$. Then  the  $\epsilon$-neighborhood of $\S_0$ is an annulus with inner radius $1$ and outer radius $1+\epsilon$.

\begin{figure}[H]
\centering
\includegraphics[scale=0.6]{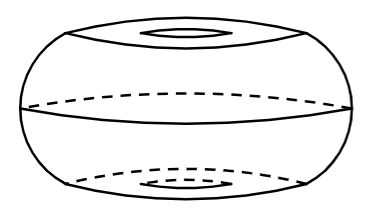}
\caption{On $[0, \epsilon]$, we have $h(r)=1+r$ and on $[L-\epsilon, L]$, we have $h(r) = -r+L+1$. This implies that the $\epsilon$-neighborhood of the boundary consists of two disjoint copies of an annulus with inner radius $1$ and outer radius $1+\epsilon$.}
\label{fig : surface de revolution}
\end{figure}

This particular case is the key idea that we use to prove Theorem \ref{thm : pricipal}. We prove it now.

\begin{proof}

We write $g_1(r,p)=dr^2+h_1^2(r)g_0(p)$. Because $h_1$ is smooth and $|h_1'| \leq 1$, we have $h_1(r) < 1+\frac{L}{2}$ for all $r \in [0, L]$. Since $h_1$ is continuous and $[0, L]$ is compact, $h_1$ reaches its maximum on $[0, L]$. We call 
\begin{align*}
    m := \max_{r \in [0, L]} \{h_1(r)\}.
\end{align*}
Notice that $1 \le m < 1+\frac{L}{2}$.
\medskip

We define a smooth function  $h_2 : [0, L] \longrightarrow \R$ by
\begin{align*}
    h_2(r)= 
    \left\{
    \begin{array}{lcl}
       1+r & \mbox{if}  & 0 \le r \le m-1 \\
    1+L-r & \mbox{if} & L-m+1 \le r \le L.      
    \end{array}
    \right.
\end{align*}
For $r\in (m-1, L-m+1)$, we only require that $h_2(r) > m$, that $h_2(L/2)= \frac{1+L/2+m}{2}$ and that 
\begin{align*}
    g_2(r,p) := dr^2 + h_2^2(r)g_0(p)
\end{align*}
is a symmetric metric of revolution on $M$, i.e for all $r \in [0, L]$, we have $h_2(r)= h_2(L-r)$. Note that we have $h_2 \ge h_1$ and that for $r \in (m-1,L-m+1)$ we have $h_2(r) > h_1(r)$.
\medskip

Besides, 
for $f$ a smooth function on $M$, we have
\begin{align*}
    R_{g_1}(f) = \frac{\int_M |\nabla f|_{g_1}^2 dV_{g_1}}{\int_\Sigma |f|^2dV_\Sigma} = \frac{\int_M \left((\partial_r f)^2 + \frac{1}{h_1^2}|\tilde{\nabla} f|_{g_0}^2 \right)h_1^{n-1}dV_{g_0}dr}{\int_\Sigma |f|^2dV_\Sigma}
\end{align*}
and 
\begin{align*}
   R_{g_2}(f) = \frac{\int_M |\nabla f|_{g_2}^2 dV_{g_2}}{\int_\Sigma |f|^2dV_\Sigma} = \frac{\int_M \left((\partial_r f)^2 + \frac{1}{h_2^2}|\tilde{\nabla} f|_{g_0}^2 \right)h_2^{n-1}dV_{g_0}dr}{\int_\Sigma |f|^2dV_\Sigma},
\end{align*}
where $\tilde{\nabla} f$ is the gradient of $f$ seen as a function of $p$. 
\medskip

Since $n\ge 3$, for all functions $f \in H^1(M)$, we have $R_{g_1}(f)\le R_{g_2}(f)$. Using the Min-Max principle, we can conclude that for all $k \ge 1$, we have $\sigma_k(M, g_1) \le \sigma_k(M, g_2)$. However, here we want to show a strict inequality.
\medskip

Because of the existence of a continuum of points $r$ for which $h_1(r) < h_2(r)$, if $\partial_r f$ does not vanish on any interval, then the inequality is strict.
\medskip


Let $k\ge 1$ be an integer. Let $E_{k+1}:=Span(f_{0,2}, \ldots, f_{k,2})$, where $f_{i,2}$ is a Steklov eigenfunction associated with $\sigma_i(M,g_2)$. We can choose these functions such that for all $i = 0, \ldots, k$, we have 
\begin{align*}
    \int_\Sigma (f_{i,2})^2 dV_\Sigma = 1,
\end{align*}
and hence
\begin{align*}
    \int_M |\nabla f_{i,2}|_{g_2}^2dV_{g_2} = \sigma_i(M,g_2).
\end{align*}
Let $f^{*} = \sum_{i=0}^k a_i f_{i,2} \in E_{k+1}$ be such that $\max_{f \in E_{k+1}} R_{g_1}(f) = R_{g_1}(f^*)$.
\medskip

We now consider two cases:
\begin{enumerate}
    \item Let us suppose $f^* = a_k f_{k,2}$ with $a_k \ne 0$, i.e $f^*$ is an eigenfunction associated with $\sigma_k(M, g_2)$. Then by \Cref{prop : sep variables}, we have $f^*(r,p) = u_j(r) S_j(p)$. Moreover, using \parencite[Prop. 2]{CV}, we know that $u_j$ is a non trivial solution of the ODE
    \begin{align*}
        \frac{1}{h^{n-1}} \frac{d}{dr} \left( h^{n-1} \frac{d}{dr}u_j  \right) - \frac{1}{h_2^2} \lambda_j u_j = 0.
    \end{align*}
    \begin{enumerate}
        \item If $\lambda_j =0$, which means $S_j = S_0 = const$, then $u_j$ cannot be locally constant. Indeed, otherwise $f^*$ would be locally constant, but since $f^*$ is harmonic, this implies that $f^*$ is constant, see \cite{Aron}. That is not the case because $k \ge 1$.
        \item If $\lambda_j \ne 0$, then $u_j$ cannot be locally constant, otherwise the ODE is not satisfied.
    \end{enumerate}
    Hence $u_j$ is not locally constant and then $\partial_r f^*$ does not vanish on any interval. Therefore, using the Min-Max principle (\ref{form : min max}), we have
    \begin{align*}
        \sigma_k(M, g_1) \le \max_{f \in E_{k+1}} R_{g_1}(f) = R_{g_1}(f^*) < R_{g_2}(f^*) = \sigma_k(M, g_2).
    \end{align*}

    \item Let us suppose $f^* = \sum_{i=0}^k a_i f_{i,2}$ such that there exists $0 \le i <k$ such that $a_i \ne 0$.
    
    Then by the Min-Max principle (\ref{form : min max}), we have
\begin{align*}
    \sigma_k(M,g_1)  & \le \max_{f \in E_{k+1}}R_{g_1}(f) 
       = R_{g_1}(f^*) 
       \le R_{g_2}(f^*) \\
      & = \frac{\int_M \sum_{i=0}^k a_i^2 |\nabla f_{i,2}|^2dV_{g_2}}{\int_\Sigma (\sum_{i=0}^k a_i f_{i,2})^2 dV_\Sigma} \\
      & = \frac{\sum_{i=0}^k a_i^2 \sigma_i(M, g_2)}{\sum_{i=0}^k a_i^2 } \quad \mbox{ since } \int_\Sigma f_{i,2} f_{j,2} dV_\Sigma = \delta_{ij} \\
      & < \sigma_k(M,g_2).
\end{align*}
\end{enumerate}

 In both cases, we have
 \begin{align*}
     \sigma_k(M, g_1) < \sigma_k(M, g_2).
 \end{align*}


\end{proof}

\begin{rem}
We never used the assumption that $g_2$ is a \textit{symmetric} metric of revolution on $M$ in the previous proof. However, it will be useful in the proofs of the theorems that follow.
\end{rem}

The process that constructs the metric $g_2$ from $g_1$ can then be repeated to create a third metric $g_3$, and so on. This generates a sequence of metrics $(g_i)$,  obtained from a sequence of functions $(h_i)$.
\begin{figure}[H]
    \centering
    \includegraphics[scale=0.6]{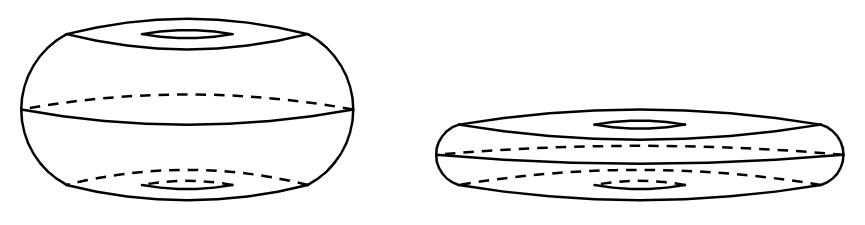}
    \caption{On the left, $M=[0,L]\times \S^{n-1}$ is endowed with a metric $g_i$ of the sequence. On the right, $M=[0,L]\times \S^{n-1}$ is endowed with another metric $g_j$ of the sequence, $j >i$.  }
    \label{fig : image suite metrique}
\end{figure}
The sequence $(h_i)$ uniformly converges to the function 
\begin{align*}
    h^* : [0, L] & \longrightarrow \R \\
    r & \longmapsto \left\{
    \begin{array}{lcl}
       1+r & \mbox{if}  & 0 \le r \le L/2 \\
    1+L-r & \mbox{if} &  L/2 \le r \le L.      
    \end{array}
    \right.
\end{align*}
This function is not smooth. Hence $(M, g^*)$, where $g^* = dr^2 + h^{*2}(r)g_0$, is not a hypersurface of revolution in the sense of Definition \ref{defn : revolution}. In the limit, $(M, g^*)$ can be seen as the gluing of two copies of an annulus of inner radius $1$ and outer radius $1+L/2$. The metric $g^*$ is therefore a maximizing metric, but is degenerated since it is induced by the function $h^*$ which is non-smooth. That is why, as already mentioned, we call $g^*$ the \textit{degenerated maximizing metric on $M$}.

\section{The first non trivial eigenvalue} \label{sect : proof 2}

In this Section, we prove Theorem \ref{thm : principal deux}. The idea consists of comparing $\sigma_1(M, g)$ with the Rayleigh quotient of a test function that is obtained from an eigenfunction for a mixed problem (Steklov-Dirichlet or Steklov-Neumann) introduced in \Cref{sect : mixed problems}. Then, to show that the upper bound $B_n(L)$ given is sharp, we take a metric of revolution $g_\epsilon$ on $M$ that is close to the degenerated maximizing metric $g^*$ and show that $\sigma_1(M, g_\epsilon)$ is close to $B_n(L)$. 
\begin{proof}
Let $(M= [0, L] \times \S^{n-1}, g)$ be a hypersurface of revolution, where $L>0$ is the meridian length of $M$. We recall that the boundary $\Sigma$ of $M$ consists of two disjoint copies of $\S^{n-1}$. We want to find a sharp upper bound $B_n(L)$ for $\sigma_1(M, g)$.
\medskip

We consider $A_{1+L/2}$ the annulus of inner radius $1$ and outer radius $1+L/2$. 
Let $\varphi_0$ be an eigenfunction for the first eigenvalue of the Steklov-Dirichlet problem on $A_{1+L/2}$, i.e.
\begin{align*}
    \sigma_0^D(A_{1+L/2}) =\frac{\int_0^{L/2} \int_{\S^{n-1}} \left( (\partial_r \varphi_0)^2 + \frac{1}{(1+r)^2} | \Tilde{\nabla}\varphi_0|^2  \right) (1+r)^{n-1}   dV_{g_0}dr}{\int_{\S^{n-1}} \varphi_0^2(0,p) dV_{g_0}}. 
\end{align*}

We define a new function 
\begin{align} \label{fct : varphi}
    \tilde{\varphi_0} : [0,L] \times \S^{n-1} & \longrightarrow \R \\ 
               (r, p) &  \longmapsto \left\{  
\begin{array}{cc}
     \varphi_0(r,p) & \mbox{ if } 0 \le r \le L/2  \\
     - \varphi_0(L-r,p) & \mbox{ if } L/2 \le r \le L. 
\end{array} \nonumber
\right.
\end{align}

The function $\tilde{\varphi_0}$ is continuous and we can check that 
\begin{align*}
    \int_\Sigma \tilde{\varphi_0}(r,p) dV_\Sigma  & = \int_{\S^{n-1}} \Tilde{\varphi_0}(0,p) dV_{g_0} + \int_{\S^{n-1}} \Tilde{\varphi_0}(L,p) dV_{g_0} \\
    & = \int_{\S^{n-1}} \varphi_0(0,p) dV_{g_0} - \int_{\S^{n-1}} \varphi_0(0,p) dV_{g_0} \\
    & = 0.
\end{align*}
Hence, thanks to formula (\ref{form : car var}), the function $\Tilde{\varphi_0}$ can be used as a test function for $\sigma_1(M,g)$.

We have 
\begin{align} \label{eqn : sigma 0}
    \sigma_1(M, g) & \le R_g(\Tilde{\varphi_0}) \nonumber \\
    & < R_{\tilde{g}}(\tilde{\varphi_0}) \quad \mbox{ where } \tilde{g}= dr^2 + \tilde{h}^2g_0 \mbox{ comes from Theorem } \ref{thm : pricipal} \nonumber \\
    & = \frac{\int_0^L \int_{\S^{n-1}} \left( (\partial_r \Tilde{\varphi_0})^2 + \frac{1}{\tilde{h}(r)^2} | \Tilde{\nabla} \Tilde{\varphi_0}|^2  \right) \tilde{h}(r)^{n-1}   dV_{g_0}dr}{\int_\Sigma \Tilde{\varphi_0}^2(0,p) dV_{\Sigma}} \nonumber \\
    & = \frac{2 \times \int_0^{L/2} \int_{\S^{n-1}} \left( (\partial_r \Tilde{\varphi_0})^2 + \frac{1}{\tilde{h}(r)^2} | \Tilde{\nabla} \Tilde{\varphi_0}|^2  \right) \tilde{h}(r)^{n-1}  dV_{g_0}dr}{2 \times \int_{\S^{n-1}} \Tilde{\varphi_0}^2(0,p) dV_{g_0}} \quad \mbox{ since } \tilde{g} \mbox{ is symmetric} \nonumber \\
    & = \frac{ \int_0^{L/2} \int_{\S^{n-1}} \left( (\partial_r \varphi_0)^2 + \frac{1}{\tilde{h}(r)^2} | \Tilde{\nabla} \varphi_0|^2  \right) \tilde{h}(r)^{n-1}   dV_{g_0}dr}{ \int_{\S^{n-1}} \varphi_0^2(0,p) dV_{g_0}} \nonumber \\
    & < \frac{\int_0^{L/2} \int_{\S^{n-1}} \left( (\partial_r \varphi_0)^2 + \frac{1}{(1+r)^2} | \Tilde{\nabla}\varphi_0|^2  \right) (1+r)^{n-1}   dV_{g_0}dr}{\int_{\S^{n-1}} \varphi_0^2(0,p) dV_{g_0}} \nonumber \\
    & = \sigma_0^D(A_{1+L/2}),
\end{align}
where the second strict inequality comes from the  existence of a continuum of points $r \in [0, L/2]$ such that $\tilde{h}(r) < 1+r$. 

\medskip

If $\phi_1$ is an eigenfunction for the first non trivial eigenvalue of the Steklov-Neumann problem on $A_{1+L/2}$, i.e 
\begin{align*}
   \sigma_1^N(A_{1+L/2}) =  \frac{\int_0^{L/2} \int_{\S^{n-1}} \left( (\partial_r \phi_1)^2 + \frac{1}{(1+r)^2} | \Tilde{\nabla} \phi_1 |^2  \right) (1+r)^{n-1}  dV_{g_0}dr}{\int_{\S^{n-1}} \phi_1^2(0,p) dV_{g_0}},
\end{align*}
then we define a new function 
\begin{align*}
    \Tilde{\phi_1} : [0, L] \times \S^{n-1} &\longrightarrow \R \\
     (r,p) & \longmapsto \left\{
     \begin{array}{cc}
         \phi_1(r,p) & \mbox{ if } 0 \le r \le L/2  \\
          \phi_1(L-r,p) & \mbox{ if } L/2 \le r \le L. 
     \end{array}
     \right.
\end{align*}
The function $\tilde{\phi_1}$ is continuous and we can check that
\begin{align*}
    \int_\Sigma \Tilde{\phi_1}(r,p) dV_\Sigma & = \int_{\S^{n-1}} \Tilde{\phi_1}(0,p) + \int_{\S^{n-1}} \Tilde{\phi_1}(L,p) \\
    & = 0 + 0 \\
    & = 0,
\end{align*}
hence we can use it as a test function for $\sigma_1(M,g)$. 
The same calculations as in (\ref{eqn : sigma 0}) show that 
\begin{align} \label{eqn : sigma 1 neumann}
    \sigma_1(M,g) < \sigma_1^N(A_{1+L/2}).
\end{align}

Putting Inequality (\ref{eqn : sigma 0}) and Inequality (\ref{eqn : sigma 1 neumann}) together, we get 
\begin{align} \label{ineg : bound}
    \sigma_1(M,g) < B_n(L) :&= \min\left\{ \sigma_0^D(A_{1+L/2}), \;  \sigma_1^N(A_{1+L/2}) \right\} \nonumber \\
    & = \min\left\{ \frac{(n-2)\left( 1+L/2 \right)^{n-2}}{\left( 1+L/2 \right)^{n-2}-1}, \frac{(n-1)\left( \left(1+L/2\right)^n -1 \right)}{\left( 1+ L/2 \right)^n +n-1}\right\}.
\end{align}

We will now prove that the bound $B_n(L)$ is sharp. This means that for each $\epsilon >0$, there exists a metric of revolution $g_\epsilon$ on $M$ such that $\sigma_1(M, g_\epsilon) > B_n(L)-\epsilon$.
\medskip

Let $\epsilon >0$. 
Let $M=[0, L] \times \S^{n-1}$ and let $g_\epsilon(r,p) = dr^2+h_\epsilon^2(r)g_0(p)$ be a metric of revolution on $M$ such that: 
\begin{enumerate}
    \item The function $h_\epsilon$ is symmetric: for all $r \in [0, L]$, we have $h_\epsilon(r)=h_\epsilon(L-r)$;
    \item For all $r \in [0, L/2- \delta]$, we have $h_\epsilon(r)=(1+r)$, with $\delta$ small enough to guarantee that for all $r \in [0, L/2]$, we have
    \begin{align*}
        \max \{ (1+r)^{n-3} - h_\epsilon(r)^{n-3}, (1+r)^{n-1}-h_\epsilon(r)^{n-1}\} < \frac{\epsilon}{B_n(L)} =: \epsilon^*.
    \end{align*}
    Geometrically, this means that $(M, g_\epsilon)$ looks like two copies of an annulus joined by a smooth curve, see Fig. \ref{fig : image suite metrique}.
\end{enumerate}

Let $f_1$ be an eigenfunction for $\sigma_1(M, g_\epsilon)$. Because $(M, g_\epsilon)$ is symmetric, then we can choose $f_1$  symmetric or anti-symmetric, which means that for all $r \in [0, L]$ and $p \in \S^{n-1}$, we have $|f_1(r,p)|=|f_1(L-r,p)|$. 

Moreover, it results from the calculations in (\ref{eqn : sigma 0}) that for any symmetric or anti-symmetric function $f$, we have
\begin{align*}
    R_{g_\epsilon}(f) =  \frac{ \int_0^{{L}/2} \int_{\S^{n-1}} \left( (\partial_r f)^2 + \frac{1}{h_\epsilon(r)^2} | \Tilde{\nabla} f|^2  \right) h_\epsilon(r)^{n-1}  dV_{g_0}dr}{ \int_{\S^{n-1}} f^2(0,p) dV_{g_0}}.
\end{align*}
We will compare  
\begin{align*}
     R_{g_\epsilon}(f_1) = \frac{ \int_0^{{L}/2} \int_{\S^{n-1}} \left( (\partial_r f_1)^2 + \frac{1}{h_\epsilon(r)^2} | \Tilde{\nabla} f_1|^2  \right) h_\epsilon(r)^{n-1}   dV_{g_0}dr}{ \int_{\S^{n-1}} (f_1)^2(0,p) dV_{g_0}}
\end{align*}
with
\begin{align*}
      R_{A_{1+L/2}}(f_1) = \frac{ \int_0^{{L}/2} \int_{\S^{n-1}} \left( (\partial_r f_1)^2 + \frac{1}{(1+r)^2} | \Tilde{\nabla} f_1|^2  \right) (1+r)^{n-1}   dV_{g_0}dr}{ \int_{\S^{n-1}} (f_1)^2(0,p) dV_{g_0}}.
\end{align*}
If we call $S :=  R_{A_{1+L/2}}(f_1) -  R_{g_\epsilon}(f_1)$, we have
\begin{align*}
    S & =  \frac{ \int_0^{L/2} \int_{\S^{n-1}}  (\partial_r f_1)^2 \left( (1+r)^{n-1} - h_\epsilon(r)^{n-1} \right) +  | \Tilde{\nabla} f_1|^2  \left( (1+r)^{n-3} -h_\epsilon(r)^{n-3} \right)   dV_{g_0}dr}{ \int_{\S^{n-1}} (f_1)^2(0,p) dV_{g_0}} \\
    & <  \frac{ \int_0^{L/2} \int_{\S^{n-1}}  ((\partial_r f_1)^2 \cdot \epsilon^* +  | \Tilde{\nabla} f_1|^2 \cdot \epsilon^* )   dV_{g_0}dr}{ \int_{\S^{n-1}} (f_1)^2(0,p) dV_{g_0}} \\
    & = \epsilon^* \cdot  \frac{ \int_0^{L/2} \int_{\S^{n-1}}  ((\partial_r f_1)^2  +  | \Tilde{\nabla} f_1|^2  )   dV_{g_0}dr}{ \int_{\S^{n-1}} (f_1)^2(0,p) dV_{g_0}} \\
    & < \epsilon^* \cdot  \frac{ \int_0^{L/2} \int_{\S^{n-1}}  ((\partial_r f_1)^2 h_\epsilon(r)^{n-1}  +  | \Tilde{\nabla} f_1|^2 h_\epsilon(r)^{n-3}  )   dV_{g_0}dr}{ \int_{\S^{n-1}} (f_1)^2(0,p) dV_{g_0}} \quad \mbox{ since $h_\epsilon \ge 1$} \\
    & = \epsilon^* \cdot \sigma_1(M, g_\epsilon) \quad \mbox{ since $f_1$ is an eigenfunction} \\
    & < \epsilon^* \cdot B_n(L) \\
    & = \epsilon.
\end{align*}
Hence, we have 
\begin{align} \label{fml : ineg sharp}
    R_{A_{1+L/2}}(f_1)  < \sigma_1(M,g_\epsilon) + \epsilon.
\end{align}
We now have two cases:
\begin{enumerate}
    \item $f_1$ can be written as $f_1(r,p)=u_0(r) S_0(p)$, where $S_0$ is a trivial harmonic function of the sphere, i.e $S_0$ is constant (we can choose $S_0 \equiv 1/\mbox{Vol}(\S^{n-1})$), and $u_0$ is smooth. Hence $f_1$ is constant on $\{0\} \times \S^{n-1}$,
    \begin{align*}
        \int_{\{0\} \times \S^{n-1}} f_1(r,p) dV_{g_0} = u_0(0) \ne 0.
    \end{align*}
    Moreover, since $|f_1(r, p)| = |f_1(L-r,p)|$ for all $r \in [0, L]$ and since
    \begin{align*}
        \int_{\Sigma} f_1(r,p) dV_\Sigma=0,
    \end{align*}
    we have 
    \begin{align*}
        f_1\left(\frac{L}{2}, p\right) = 0.
    \end{align*}
    Therefore, we can use $f_{1{\vert_{[0, L/2] \times \S^{n-1}}}}$ as a test function for $\sigma_0^D(A_{1+L/2})$, and we can state 
\begin{align*}
    \sigma_0^D(A_{1+L/2}) \leq R_{A_{1+L/2}}(f_1).
\end{align*}

    \item $f_1$ can be written as $f_1(r,p) = u_1(r) S_1(p)$, where $S_1$ is a non constant harmonic function of the sphere associated with the first non zero eigenvalue  and $u_1$ is smooth.  Hence 
    \begin{align*}
        \int_{\{0\} \times \S^{n-1}} f_1(r,p) dV_{g_0} = 0.
    \end{align*}
    Moreover, we have $u_1(L/2) > 0$. 
\medskip    

    Added with the fact that $|f_1(r, p)| = |f_1(L-r,p)|$ for all $r \in [0, L]$ and because $f_1$ is smooth, we can conclude
    \begin{align*}
        \partial_r f_1\left( \frac{L}{2}, p\right) = 0.
    \end{align*}

    Therefore, we can use $f_{1{\vert_{[0, L/2] \times \S^{n-1}}}}$ as a test function for $\sigma_1^N(A_{1+L/2})$ and we can state 
\begin{align*}
    \sigma_1^N(A_{1+L/2}) \leq R_{A_{1+L/2}}(f_1).
\end{align*}
\end{enumerate}
But we defined $B_n(L)$ as  
\begin{align*}
      B_n(L) = \min \{\sigma_0^D(A_{1+L/2}), \sigma_1^N(A_{1+L/2}) \}.
\end{align*}
Hence we have 
\begin{align*}
    B_n(L) \le R_{A_{1+L/2}}(f_1)  \stackrel{(\ref{fml : ineg sharp})}{<} \sigma_1(M,g_\epsilon) + \epsilon 
\end{align*}
and then
\begin{align*}
    \sigma_1(M,g_\epsilon) > B_n(L) - \epsilon.
\end{align*}
\end{proof}

From this result we can prove Corollary \ref{cor : corollaire du principal deux}.

\begin{proof}

By \Cref{thm : principal deux}, the inequality  (\ref{ineg : bound}) holds which is
\begin{align*}
    \sigma_1(M,g) <  \min\left\{ \frac{(n-2)\left( 1+L/2 \right)^{n-2}}{\left( 1+L/2 \right)^{n-2}-1}, \frac{(n-1)\left( \left(1+L/2\right)^n -1 \right)}{\left( 1+ L/2 \right)^n +n-1}\right\}.
\end{align*}

We consider the two functions 
\begin{align*}
    L & \longmapsto \frac{(n-2)\left( 1+L/2 \right)^{n-2}}{\left( 1+L/2 \right)^{n-2}-1} = \sigma_0^D(A_{1+L/2}) \\
    L & \longmapsto \frac{(n-1)\left( \left(1+L/2\right)^n -1 \right)}{\left( 1+ L/2 \right)^n +n-1} = \sigma_1^N(A_{1+L/2}).
\end{align*}

\begin{figure}[H]
\centering
\includegraphics[scale=1.0]{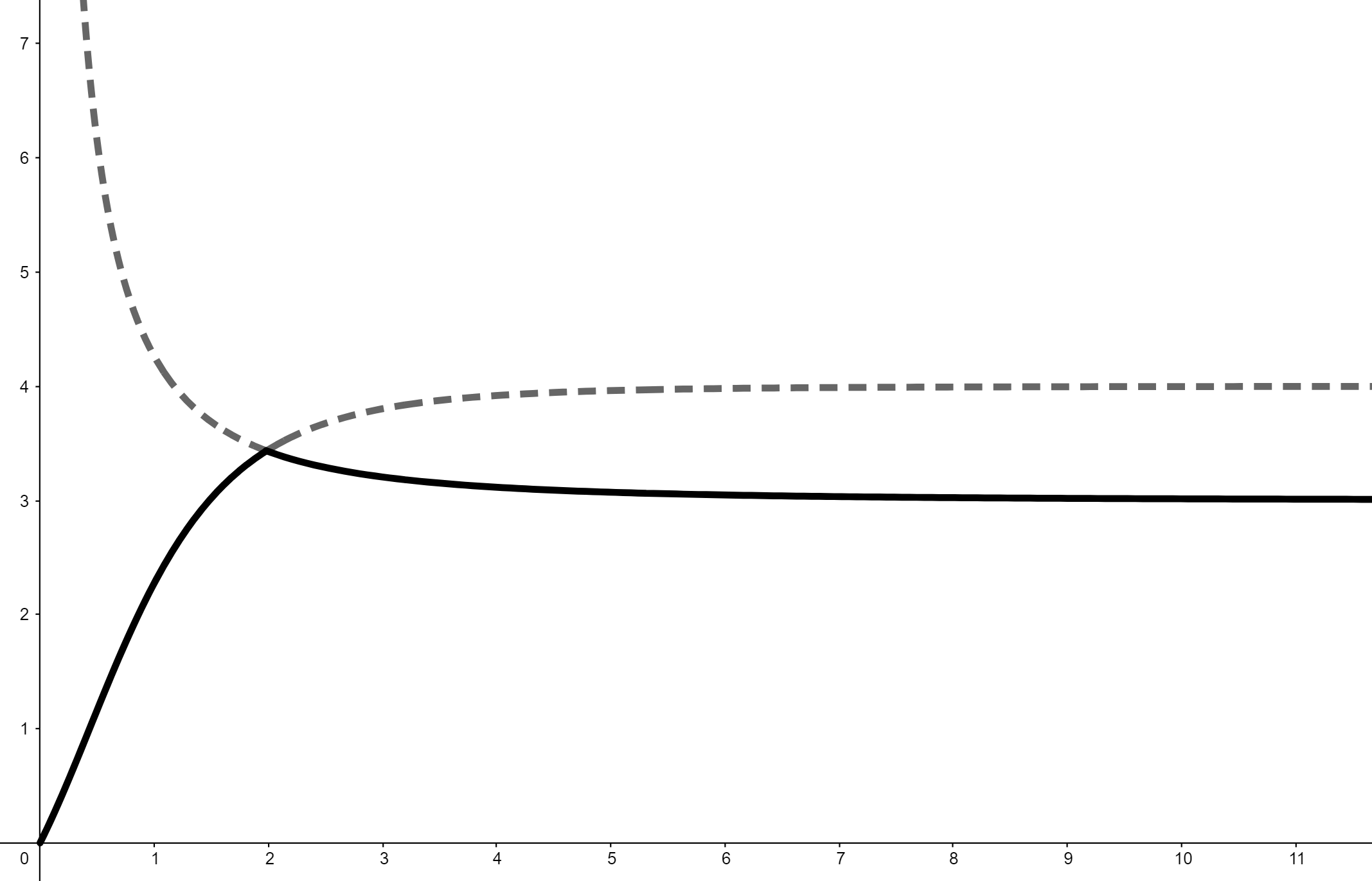}
\caption{Representation of the case $n = 5$. The decreasing smooth curve is $L \longmapsto \sigma_0^D(A_{1+L/2})$ while the increasing smooth curve is $L \longmapsto \sigma_1^N(A_{1+L/2})$. The solid curve is the bound $B_5(L)$ given by Theorem \ref{thm : principal deux}. 
}
\label{fig : borne}
\end{figure}

We can show that $L \longmapsto \sigma_0^D(A_{1+L/2})$ is strictly decreasing with $L$. Indeed, let $L' > L$ and let $\varphi_0$ be an eigenfunction for $\sigma_0^D(A_{1+L/2})$. We consider 
\begin{align*}
    \bar{\varphi}_0 : [0, \frac{L'}{2}] \times \S^{n-1} & \longrightarrow \R 
\end{align*}
the extension by $0$ of $\varphi_0$ to the annulus $A_{1+L'/2}$. We get
\begin{align*}
    \sigma_0^D(A_{1+L'/2}) < R_{A_{1+L/2}}(\Bar{\varphi}_0) =  R_{A_{1+L'/2}}(\varphi_0)  = \sigma_0^D(A_{1+L/2}),
\end{align*}
where the strict inequality comes from the fact that $\Bar{\varphi}_0$ is not an eigenfunction associated with $\sigma_0^D(A_{1+L'/2})$. Indeed, if we suppose that $\Bar{\varphi}_0$ is an eigenfunction for $\sigma_0^D(A_{1+L'/2})$, then it is harmonic in $A_{1+L'/2}$ (since it satisfies the Steklov-Dirichlet problem), and since $\Bar{\varphi}_0$ vanishes on the open set $A_{1+L'/2} \backslash A_{1+L/2}$, then by \cite{Aron} $\Bar{\varphi}_0$ is constant, which is a contradiction.
\medskip

In the same way, we can show that   $L \longmapsto \sigma_1^N(A_{1+L/2})$ is strictly increasing with $L$.  Indeed, let $L' > L$ and let $\phi_1$ be an eigenfunction for $\sigma_1^N(A_{1+L'/2})$. We consider 
\begin{align*}
     \bar{\phi}_1 : [0, \frac{L}{2}] \times \S^{n-1} & \longrightarrow \R 
\end{align*}
the restriction of $\phi_1$ to the annulus $A_{1+L/2}$. We get
\begin{align*}
    \sigma_1^N(A_{1+L/2}) \le R_{A_{1+L/2}}(\Bar{\phi}_1) < R_{A_{1+L'/2}}(\phi_1) = \sigma_1^N(A_{1+L'/2}).
\end{align*}

Hence the bound we gave possesses a maximum depending only on the dimension $n$, given by
\begin{align*} 
    \sigma_1(M,g) < B_n := \frac{(n-2)\left( 1+L_1/2 \right)^{n-2}}{\left( 1+L_1/2 \right)^{n-2}-1}
\end{align*}
where $L_1$ is the unique positive solution of the equation
\begin{align*}
    \left(1+L/2\right)^{2n-2}-(n-1)\left(1+L/2\right)^n-(n-1)^2\left(1+L/2)\right)^{n-2}+n-1=0.
\end{align*}

In order to prove that this bound is sharp, let $\epsilon >0$. We define $M_\epsilon := [0, L_1]\times \S^{n-1}$. Theorem \ref{thm : principal deux} guarantees that there exists a metric of revolution $g_\epsilon$ on $M_\epsilon$ such that $\sigma_1(M_\epsilon, g_\epsilon) > B_n(L_1) - \epsilon = B_n - \epsilon$, which ends the proof.

\end{proof}

We continue by proving \Cref{rem : dimension sur L1}.

\begin{proof}
We know that there exists a unique positive value of $L$, that we call $L_1=L_1(n)$, such that the equality 
\begin{align*}
    \left(1+L/2\right)^{2n-2}-(n-1)\left(1+L/2\right)^n-(n-1)^2\left(1+L/2)\right)^{n-2}+n-1=0
\end{align*}
holds. To ease notation, we substitute $(1+L/2)$ by $R$ and we can state that there is a unique value of $R \in (1, \infty)$ such that the equality 
\begin{align*}
    R^{2n-2}-(n-1)R^n-(n-1)^2R^{n-2}+n-1=0
\end{align*}
holds. This equation is equivalent to
\begin{align*}
    R^{n-2}\left( R^{n}-(n-1)R^2-(n-1)^2\right)+n-1=0,
\end{align*}
and we call $R_1=R_1(n)$ its unique solution in $(1, \infty)$.
We prove that $R_1(n) \underset{n \to \infty}{\longrightarrow} 1$.

\medskip
We call 
\begin{align*}
    \psi_n(R) :=  R^{n}-(n-1)R^2-(n-1)^2
\end{align*}
and 
\begin{align*}
    \Psi_n(R) :=  R^{n-2}\left( R^{n}-(n-1)R^2-(n-1)^2\right)+n-1.
\end{align*}

Then, for $R_1$ to be such that $\Psi_n(R_1) =0$, it is necessary that $\psi_n(R_1)<0$.
\medskip

Thus, 
\begin{align*}
    R_1^{n} & < (n-1)R_1^2 + (n-1)^2 \\
    & < (n-1)^2(R_1^2 +1) \\
    & < (n-1)^2 \cdot 2R_1^2 \hspace{2 cm} \mbox{ since } R_1 > 1 \\
    & < (n-1)^3 R_1^2 \hspace{2 cm} \mbox{ since } n-1 \ge 2.
\end{align*}
Therefore,
\begin{align*}
    n \ln(R_1) < 3 \ln(n-1) + 2 \ln(R_1) 
\end{align*}
so
\begin{align*}
    \ln(R_1) < \frac{3 \ln(n-1)}{n-2}
\end{align*}
and 
\begin{align*}
    R_1 < e^{\frac{3 \ln(n-1)}{n-2}}. 
\end{align*}
As we substituted $(1+L/2)$ by $R$, and we can state that
\begin{align*}
    L_1(n) & < 2 \left( e^{\frac{3 \ln(n-1)}{n-2}} -1 \right). \\
\end{align*}
Therefore, since $\frac{3 \ln(n-1)}{n-2} \underset{n \to \infty }{\longrightarrow} 0$, we have
\begin{align*}
    L_1(n)  \underset{n \to \infty }{ \longrightarrow } 0.
\end{align*}

Moreover, we have 
\begin{align*}
    n-1 > B_n 
    > n-2 \underset{n \to \infty}{\longrightarrow} \infty.
\end{align*}

\end{proof}

\section{Stability properties of hypersurfaces of revolution } \label{sect : stability}

The goal of this section is to prove Theorem \ref{thm : stability 1} and Theorem \ref{thm : stability k}, which show some stability properties of the hypersurfaces we are studying in this paper. For \Cref{thm : stability 1}, the key idea is to choose $L \ne L_1$ and compare $\sigma_1(M = [0, L] \times \S^{n-1}, g)$ with the first non trivial eigenvalue of $M$ when endowed with the degenerated maximizing metric, namely $B_n(L)$. For the case of \Cref{thm : stability k}, the strategy consists of showing that among all metrics of revolution that are not close (in a sense properly defined) to the degenerated maximizing metric, none of them induces a first non trivial eigenvalue that is close to $B_n(L)$. We prove these theorems now.

\subsection{Proof of  Theorem \ref{thm : stability 1}}

Recall that here we suppose $L \ne L_1$.

\begin{proof}
Let $g$ be any metric of revolution on $M= [0, L] \times \S^{n-1}$. Then we have 
\begin{align*}
    \sigma_1(M, g) < B_n(L),
\end{align*}
where $B_n(L)$ is given by \Cref{thm : principal deux}.
\medskip

We define $C(n, L):= B_n - B_n(L)$, which is strictly positive since we assumed $L \ne L_1$. Then we have 
\begin{align*}
    B_n - \sigma_1(M, g) \ge B_n - B_n(L) = C(n, L).
\end{align*}

Let $0< \delta < \frac{B_n-(n-2)}{2}$, and let us suppose $|B_n - \sigma_1(M, g)| < \delta$. Therefore, we have
$|B_n - \sigma_1(M, g^*)| < \delta$, where we wrote $g^*$ the degenerated maximizing metric on $M$.  We consider two cases:
\begin{enumerate}
    \item We suppose $L_1 < L$. In this case, we have $B_n(L) = \sigma_0^D(A_{1+L/2}) = \frac{(n-2)(1+L/2)^{n-2}}{(1+L/2)^{n-2}-1}$. We write 
    \begin{align*}
        R := 1+L/2 \mbox{ and } \sigma_1(R) := \frac{(n-2)R^{n-2}}{R^{n-2}-1}.
    \end{align*}
    Hence we have $|B_n - \sigma_1(R)| < \delta \implies R \in [R_1, R_\delta]$, where $R_1 = 1+L_1/2$ and $R_\delta$ is defined by    $\sigma_1(R_\delta) = B_n - \delta$. Note that $R_\delta$ exists since we assumed $\delta < B_n -(n-2)$. We can calculate that 
    \begin{align*}
        R_\delta = \left( \frac{B_n - \delta}{B_n - (n-2) - \delta}\right)^{\frac{1}{n-2}} \mbox{ and } R_1 = \left( \frac{B_n}{B_n - (n-2)} \right)^{\frac{1}{n-2}}.
    \end{align*}
    Thus, we have 
    \begin{align*}
        |R_1 - R| \le R_\delta - R_1 = \left( \frac{B_n - \delta}{B_n - (n-2) - \delta}\right)^{\frac{1}{n-2}} - \left( \frac{B_n}{B_n - (n-2)} \right)^{\frac{1}{n-2}}.
    \end{align*}
    To estimate this expression, we use the identity $x^{n-2} - y^{n-2} = (x-y)(x^{n-3} + x^{n-4}y + \ldots + xy^{n-4} + y^{n-3})$, with $x = R_\delta$ and $y = R_1$. On the one hand, we can compute that 
    \begin{align*}
        R_\delta^{n-2} - R_1^{n-2} = \frac{(n-2)\delta}{(B_n-(n-2)-\delta)(B_n-(n-2))} \le \frac{2(n-2)\delta}{(B_n - (n-2))^2},
    \end{align*}
    where the inequality comes from the assumption $\delta < \frac{B_n - (n-2)}{2}$. On the other hand, we can compute that 
    \begin{align*}
        R_\delta^{n-3} + R_\delta^{n-4}R_1 + \ldots + R_\delta R_1^{n-4} + R_1^{n-3} \ge (n-2) \cdot \left(  \frac{B_n}{B_n -(n-2)}\right)^{\frac{n-3}{n-2}}.
    \end{align*}
    Therefore, 
    \begin{align*}
        R_\delta - R_1   \le \frac{2/(B_n-(n-2))^2}{\left( B_n/(B_n-(n-2))\right)^{\frac{n-3}{n-2}}} \cdot \delta := C_1(n) \cdot \delta.
    \end{align*}
    Since we wrote $R = 1+L/2$, we can conclude that, for $L_1 < L$ and $0 < \delta < \frac{B_n -(n-2)}{2}$, we have 
    \begin{align*}
        B_n - \sigma_1(M, g) < \delta \implies L - L_1 < 2C_1(n) \cdot \delta.
    \end{align*}

    \item Now we suppose $L < L_1$ and we do a similar calculation, this time with $B_n(L) = \sigma_1^N(A_{1+L/2}) = \frac{(n-1)((1+L/2)^n-1)}{(1+L/2)^n+n-1}$. We obtain a constant 
    \begin{align*}
        C_2(n) := \frac{ (n-2)^2/(n-1-B_n)^2  }{ n \left(  ((n-1)B_n+1)/(n-1-B_n)    \right)^{ \frac{1}{n}  }}
    \end{align*}
    such that 
    \begin{align*}
        B_n -\sigma_1(M, g) < \delta \implies |L_1 -L| \le 2C_2(n) \cdot \delta.
    \end{align*}
\end{enumerate}
Defining $C(n) := 2 \cdot \max \{ C_1(n), C_2(n)\}$ concludes the proof.

\end{proof}

\subsection{Proof of Theorem \ref{thm : stability k}}

Recall that we fixed $m \in [1, 1+L/2)$ and that we defined $\mathcal{M}_m := \{$metrics of revolution $g$ induced by a function $h$ such that $\max_{r \in [0, L]} \{h(r)\} \le m\}$.  

\begin{proof}
Let $g \in \mathcal{M}_m$, and let $h : [0, L] \longrightarrow \R_+^*$ be the function which induces $g$. We define a new function $h_m : [0, L] \longrightarrow \R_+^*$ as follows:
    \begin{align*}
    h_m(r) = \left\{
    \begin{array}{ll}
        1+r & \mbox{ if } 0 \le r \le m-1  \\
        m & \mbox{ if } m-1 \le r \le L -m+1 \\
       1+L-r  & \mbox{ if } L-m+1 \le r \le L. 
    \end{array}
    \right.
\end{align*}
\begin{figure}[ht]
    \centering
    \includegraphics[scale=0.5]{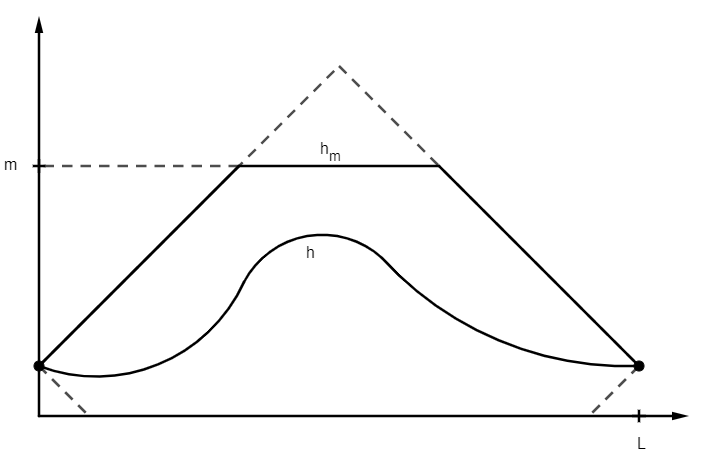}
    \caption{Since $g \in \mathcal{M}_m$, the function $h$ which induces $g$ satisfies $h \le h_m$.}
    \label{fig: h_m}
\end{figure}

We call $g_m$ the metric induced by $h_m$. Notice that $g_m$ is not a metric of revolution in the sense of  \Cref{defn : revolution} since $h_m$ is not smooth.
\medskip

In the same spirit as  in Sect. \ref{sect : proof},  
for any smooth function $f$ on $M$, we have
\begin{align*}
    R_{g}(f) =  \frac{\int_M \left((\partial_r f)^2 + \frac{1}{h^2}|\tilde{\nabla} f|_{g_0}^2 \right)h^{n-1}dV_{g_0}dr}{\int_\Sigma |f|^2dV_\Sigma}
\end{align*}
and 
\begin{align*}
   R_{g_m}(f) = \frac{\int_M \left((\partial_r f)^2 + \frac{1}{h_m^2}|\tilde{\nabla} f|_{g_0}^2 \right)h_m^{n-1}dV_{g_0}dr}{\int_\Sigma |f|^2dV_\Sigma}.
\end{align*}
Therefore, since $n \ge 3$ and $h \le h_m$,  we have 
\begin{align*}
    \sigma_1(M, g) \le \sigma_1(M, g_m).
\end{align*}
We can now consider a new function $\tilde{h}_m$, obtained from $h_m$ by smoothing out the two non-smooth points, with $\tilde{h}_m$ satisfying:
\begin{enumerate}
    \item For all $r \in [0, L]$, we have $h_m(r) \le \tilde{h}_m(r)$;
    \item  The metric $\tilde{g}_m$ induced by $\tilde{h}_m$ is a metric of revolution in the sense of \Cref{defn : revolution}.
\end{enumerate}
Remark that since $h_m \le \tilde{h}_m$, we have $\sigma_1(M, g_m) \le \sigma_1(M, \tilde{g}_m)$.
\medskip

We define $C(n, L, m) := B_n(L) - \sigma_1(M, \tilde{g}_m)$, which is strictly positive by \Cref{thm : principal deux}. Then we have
\begin{align*}
    B_n(L)- \sigma_1(M, g) \ge B_n(L)- \sigma_1(M, \tilde{g}_m) = C(n, L, m).
\end{align*}
\end{proof}





\section{Upper bounds for higher Steklov eigenvalues} \label{sect : 6}

In this section, we want to compute some sharp upper bound for higher Steklov eigenvalues of hypersurfaces of revolution. Therefore, we will have to deal with the multiplicity of the eigenvalues. We write $\lambda_{(k)}, \sigma_{(k)}, \sigma_{(k)}^D, \sigma_{(k)}^N$ for the $(k)$th eigenvalue counted without multiplicity.
\medskip

Before we can state and prove our results, we first recall some known properties of the multiplicities of the eigenvalues under consideration.
\medskip


Given a hypersurface of revolution $(M=[0,L] \times \S^{n-1}, g)$, we want to provide information about the multiplicity of the Steklov eigenvalues of $(M, g)$.
\medskip

For the classical Laplacian problem $\Delta S = \lambda S$ on  $(\S^{n-1}, g_0)$, we know \parencite[pp. 160-162]{BGM} that the set of eigenvalues is $\{ \lambda_{(k)} = k(n+k-2) : k \ge 0 \}$, where the multiplicity $m_0$ of $\lambda_{(0)}=0$ is $1$ and the multiplicity of $\lambda_{(k)}$ is 
\begin{align} \label{frm : multiplicite}
    m_k:=\frac{(n+k-3)(n+k-4) \ldots n(n-1)}{k!}(n+2k-2).
\end{align}
 
As such, given $k \ge 0$, there exist $m_k$ independent functions $S_k^1, \ldots, S_k^{m_k}$ such that $\Delta S_k^i = \lambda_{(k)} S_k^i, \; i=1, \ldots, m_k$.
\medskip

Given $k \ge 0$, there are $m_k$ independent Steklov-Dirichlet eigenfunctions associated with the eigenvalue $\sigma_{(k)}^D(A_{1+L/2})$, that can be written $\varphi_k^i(r,p) = \alpha_{k}(r)S_{k}^i(p), \; i = 1, \ldots, m_k$. For the Steklov-Neumann case, the eigenfunctions associated with $\sigma_{(k)}^N(A_{1+L/2})$  can be written $\phi_k^i(r,p) = \beta_{k}(r)S_k^i(p), \; i = 1, \ldots, m_k$. Indeed, for each of these problems, the multiplicity of the $(k)$th eigenvalue is exactly $m_k$, see, for example, \parencite[Prop. 3]{CV}.

\subsection{Upper bound for \texorpdfstring{$\sigma_{2}(M, g), \ldots, \sigma_{m_1}(M, g)$}{sigma2}} \label{subsect : sigma2 sigmam1}

In this section, we prove the following theorem:
\begin{thm} \label{thm : bound sigma2 sigmam1}
    Let $(M = [0, L] \times \S^{n-1}, g)$ be a hypersurface of revolution in Euclidean space with two boundary components each isometric to $\S^{n-1}$ and dimension $n \ge 3$. Let $m_1$ be the multiplicity of the first non trivial eigenvalue of the classical Laplacian problem on $(\S^{n-1}, g_0)$. Then we have 
    \begin{align*}
        \sigma_2(M, g) = \ldots = \sigma_{m_1}(M, g) < B_n^2(L) = \ldots = B_n^{m_1}(L) = : \sigma_{(1)}^N(A_{1+L/2}).
    \end{align*}
    Moreover, this bound is sharp: for all $\epsilon >0$ there exists a metric of revolution $g_\epsilon$ on $M$ such that 
    \begin{align*}
        \sigma_2(M, g_\epsilon) = \ldots = \sigma_{m_1}(M, g_\epsilon) > \sigma_{(1)}^N(A_{1+L/2}) - \epsilon.
    \end{align*}
\end{thm}

\begin{proof}
We consider two cases. 
\begin{enumerate}
    \item Let $M=[0,L] \times \S^{n-1}$, with $L\le L_1$. We write $f_1^1$ an eigenfunction associated with $\sigma_1(M,g)$. Since $L \le L_1$, we have $B_n(L)=\sigma_1^N(A_{1+L/2})$ and therefore $f_1^1(r,p)=u_1(r)S_1^1(p)$.

    We consider now a new function denoted $f_1^2$ given by
    $
        f_1^2(r,p)=u_1(r)S_1^2(p).
    $
    We can check that 
    \begin{align*}
        \int_\Sigma f_1^2(r,p)dV\Sigma=0 \quad \mbox{ and } \quad \int_\Sigma f_1^1(r,p)f_1^2(r,p)dV\Sigma =0.
    \end{align*}
    Moreover, we have 
    \begin{align*}
        \sigma_1(M, g)=R_g(f_1^1) = R_g(f_1^2).
    \end{align*}
    In the same way, we write 
    \begin{align*}
        f_1^i(r,p) = u_1(r) S_1^i(p), \quad i=1, \ldots, m_1
    \end{align*}
    and we can conclude
    \begin{align*}
        \sigma_1(M,g) = \sigma_2(M,g) = \ldots = \sigma_{m_1}(M, g).
    \end{align*}
    Therefore, we already have a sharp upper bound for these eigenvalues, which is given by $\sigma_1^N(A_{1+L/2})$.
    \item Let $M=[0,L] \times \S^{n-1}$, with $L>L_1$.  
    We call $f_1$ an eigenfunction associated with $\sigma_1(M,g)$. Since $L > L_1$, we have $B_n(L)=\sigma_0^D(A_{1+L/2})$. Therefore $f_1(r,p)=u_0(r)S_0(p)$.
    
   We write now $f_2^1(r,p) = u_2(r)S_1^1(p)$ an eigenfunction associated with $\sigma_2(M,g)$. As before, we then consider $m_1$ functions denoted $f_2^i(r,p) = u_2(r)S_1^i(p), i= 1, \ldots, m_1$ and we get 
    \begin{align*}
    \sigma_2(m,g) = \ldots = \sigma_{m_1+1}(M,g).
    \end{align*}
     We consider a function $\phi_1(r,p) = \beta_1(r)S_1(p)$ associated with $\sigma_{(1)}^N(A_{1+L/2})$. In the same spirit as before, we define a function
    \begin{align*}
    \Tilde{\phi_1} : [0, L] \times \S^{n-1} &\longrightarrow \R \\
     (r,p) & \longmapsto \left\{
     \begin{array}{cc}
         \phi_1(r,p) & \mbox{ if } 0 \le r \le L/2  \\
          \phi_1(L-r,p) & \mbox{ if } L/2 \le r \le L. 
     \end{array}
     \right.
\end{align*}
We can check that the function $\tilde{\phi_1}$ is continuous and that $\int_\Sigma \tilde{\phi_1}dV_\Sigma=0$. Moreover, it is immediate that $\int_\Sigma \tilde{\phi_1} f_1 dV_\Sigma = 0$. Hence we can use $\tilde{\phi_1}$ as a test function for $\sigma_2(M, g)$ and as we did before, we can see that
\begin{align*} 
    \sigma_2(M, g) & \le R_g(\Tilde{\phi_1}) \nonumber \\
    & < R_{\tilde{g}}(\tilde{\phi_1}) \nonumber \mbox{ where } \tilde{g} \mbox{ comes from Theorem \ref{thm : pricipal}} \\
    & = \frac{\int_0^L \int_{\S^{n-1}} \left( (\partial_r \Tilde{\phi_1})^2 + \frac{1}{\tilde{h}(r)^2} | \Tilde{\nabla} \Tilde{\phi_1}|^2  \right) \tilde{h}(r)^{n-1}   dV_{g_0}dr}{\int_\Sigma \Tilde{\phi_1}^2(0,p) dV_{g_0}} \nonumber \\
    & = \frac{2 \times \int_0^{L/2} \int_{\S^{n-1}} \left( (\partial_r \Tilde{\phi_1})^2 + \frac{1}{\tilde{h}(r)^2} | \Tilde{\nabla} \Tilde{\phi_1}|^2  \right) \tilde{h}(r)^{n-1}  dV_{g_0}dr}{2 \times \int_{\S^{n-1}} \Tilde{\phi_1}^2(0,p) dV_{g_0}} \nonumber \\
    & = \frac{ \int_0^{L/2} \int_{\S^{n-1}} \left( (\partial_r \phi_1)^2 + \frac{1}{\tilde{h}(r)^2} | \Tilde{\nabla} \phi_1|^2  \right) \tilde{h}(r)^{n-1}   dV_{g_0}dr}{ \int_{\S^{n-1}} \phi_1^2(0,p) dV_{g_0}} \nonumber \\
    & < \frac{\int_0^{L/2} \int_{\S^{n-1}} \left( (\partial_r \phi_1)^2 + \frac{1}{(1+r)^2} | \Tilde{\nabla}\phi_1|^2  \right) (1+r)^{n-1}   dV_{g_0}dr}{\int_{\S^{n-1}} \phi_1^2(0,p) dV_{g_0}} \nonumber \\
    & = \sigma_{(1)}^N(A_{1+L/2}).
\end{align*}
\end{enumerate}

Therefore, regardless of the value of $L>0$, we have 
\begin{align*}
    \sigma_2(M,g) = \ldots = \sigma_{m_1}(M,g) < B_n^2(L) = \ldots = B_n^{m_1}(L) := \sigma_{(1)}^N(A_{1+L/2}).
\end{align*}
Moreover, this bound is sharp : for all $\epsilon >0$, there exists a metric $g_\epsilon$ on $M=[0,L]\times \S^{n-1}$ such that $\sigma_2(M, g_\epsilon) = \ldots = \sigma_{m_1}(M,g_\epsilon) > \sigma_{(1)}^N(A_{1+L/2}) - \epsilon$. Indeed, as before it is sufficient to choose the metric $g_\epsilon = dr^2 + h_\epsilon^2g_0$, with the function $h_\epsilon$ such that 
\begin{enumerate}
    \item $h_\epsilon$ is symmetric;
    \item For all $r \in [0, L/2- \delta]$, we have $h_\epsilon(r)=1+r$, with $\delta$ small enough. 
\end{enumerate}
The proof of sharpness goes as in the proof of \Cref{thm : principal deux}.

\end{proof}

The upper bound we gave, namely $ \sigma_{(1)}^N(A_{1+L/2})$, depends on the dimension of $M$ and the meridian length $L$ of $M$. It is easy to see that $\sigma_{(1)}^N(A_{1+L/2})$, which is strictly increasing, satisfies
\begin{align*}
    \sigma_{(1)}^N(A_{1+L/2}) = \frac{(n-1)\left( \left(1+L/2\right)^n -1 \right)}{\left( 1+ L/2 \right)^n +n-1} \underset{L \rightarrow \infty}{\longrightarrow} n-1.
\end{align*}
Therefore, we have got a bound that depends only on the dimension $n$ of $M$. Given a hypersurface of revolution $(M, g)$ with two boundary components, we have 
\begin{align*}
    \sigma_2(M,g) = \ldots = \sigma_{m_1}(M,g) < B_n^2 = \ldots = B_n^{m_1} := n-1.
\end{align*}
Moreover, this bound is sharp, in the sense that for all $\epsilon >0$, there exists a hypersurface of revolution $(M_\epsilon, g_\epsilon)$ such that $\sigma_2(M_\epsilon,g_\epsilon) = \ldots = \sigma_{m_1}(M_\epsilon,g_\epsilon) > n-1 - \epsilon$. 
Indeed, we can choose $L_\epsilon$ large enough for $\sigma_1^N(A_{1+L_\epsilon/2})$ to be $\frac{\epsilon}{2}$-close to $n-1$, and then define $M_\epsilon := [0,L_\epsilon] \times \S^{n-1}$. Now we can put a metric $g_\epsilon$ on $M_\epsilon$ such that $\sigma_2(M_\epsilon, g_\epsilon) = \ldots = \sigma_{m_1}(M_\epsilon, g_\epsilon) > \sigma_1^N(A_{1+L_\epsilon/2}) - \frac{\epsilon}{2}$, and we are done.
\medskip

Our calculations showed that the eigenvalues $k = 2, \ldots, m_1$ have a critical length at infinity.

\subsection{Upper bound for \texorpdfstring{$\sigma_{m_1+1}(M,g)$}{sigmam1plus1}} \label{subsect : sigmam1plus1}

Now we are interested in the next eigenvalue, namely $\sigma_{m_1+1}(M, g)$. For that reason, we define a new special meridian length $L_2$: it is the unique solution of the equation $\sigma_0^D(A_{1+L/2}) = \sigma_{(2)}^N(A_{1+L/2})$. We remark that we have $L_2<L_1$. Indeed, for all $L >0$, we have $\sigma_{(2)}^N(A_{1+L/2}) > \sigma_{(1)}^N(A_{1+L/2})$, and the function $L \longmapsto \sigma_0^D(A_{1+L/2})$ is strictly decreasing. Hence, comparing the intersection of the curves gives $L_2 < L_1$.  We prove the following theorem:
\begin{thm} \label{thm : bound sigmam1plus1}
    Let $(M = [0, L] \times \S^{n-1}, g)$ be a hypersurface of revolution in Euclidean space with two boundary components each isometric to  $\S^{n-1}$ and dimension $n \ge 3$. Let $m_1$ be the multiplicity of the first non trivial eigenvalue of the classical Laplacian problem on $(\S^{n-1}, g_0)$. Then we have
    \begin{align*}
    \sigma_{m_1+1}(M,g) < B_n^{m_1+1}(L) := \left\{
    \begin{array}{ll}
         \sigma_{(2)}^N(A_{1+L/2}) & \mbox{ if } L \le L_2  \\
         \sigma_{(0)}^D(A_{1+L/2}) & \mbox{ if } L_2 < L \le L_1 \\
         \sigma_{(1)}^N(A_{1+L/2}) & \mbox{ if } L_1 < L.
    \end{array}
    \right.
\end{align*}
Moreover, this bound is sharp: for all $\epsilon >0$, there exists a metric of revolution $g_\epsilon$ on $M$ such that 
\begin{align*}
    \sigma_{m_1+1}(M, g_\epsilon) > B_n^{m_1+1}(L) - \epsilon.
\end{align*}
\end{thm}

A plot of the function $L \longmapsto B_n^{m_1+1}(L)$ can be useful to visualize the sharp upper bound, see Fig. \ref{fig : sigmam1plus1}.

\begin{proof}
Now we have to distinguish three cases.
\begin{enumerate}
    \item Let $M=[0,L]\times \S^{n-1}$, with $L \le L_2$. We call $f_1^1(r,p) = u_1(r)S_1^1(p), \ldots, f_1^{m_1}(r,p)= u_1(r)S_1^{m_1}(p)$ the Steklov eigenfunctions associated with $\sigma_{(1)}(M,g)= \sigma_1(M,g) = \ldots = \sigma_{m_1}(M,g)$. 
    
    There exists an eigenfunction $\phi_2(r,p) = \beta_2(r) S_2(p)$ associated with $\sigma_{(2)}^N(M,g)= \sigma_{m_1+1}^N(M,g)$. We define a new function
    \begin{align*}
    \Tilde{\phi_2} : [0, L] \times \S^{n-1} &\longrightarrow \R \\
     (r,p) & \longmapsto \left\{
     \begin{array}{cc}
         \phi_2(r,p) & \mbox{ if } 0 \le r \le L/2  \\
          \phi_2(L-r,p) & \mbox{ if } L/2 \le r \le L. 
     \end{array}
     \right.
\end{align*}
This function is continuous, satisfies $\int_\Sigma \tilde{\phi_2}dV_\Sigma =0$ and we can check that for all $i=1, \ldots, m_1$,
\begin{align*}
    \int_\Sigma \tilde{\phi_2} f_1^idV_\Sigma =0.
\end{align*}
Hence we can use $\tilde{\phi_2}$ as a test function for $\sigma_{m_1+1}(M,g)$. The same kind of calculations as in Inequality (\ref{eqn : sigma 1 neumann}) show that we have 
\begin{align*}
    \sigma_{m_1+1}(M, g) < \sigma_{(2)}^N(A_{1+L/2}),
\end{align*}
which is a sharp upper bound.
    \item Let $M=[0,L]\times \S^{n-1}$, with $L_2 < L \le L_1$.  We call $f_1^1(r,p) = u_1(r)S_1^1(p), \ldots, f_1^{m_1}(r,p)= u_1(r)S_1^{m_1}(p)$ the Steklov eigenfunctions associated with $\sigma_{(1)}(M,g)= \sigma_1(M,g) = \ldots = \sigma_{m_1}(M,g)$.

 There exists an eigenfunction $\varphi_0(r,p) = \alpha_0(r) S_0(p)$ associated with $\sigma_{0}^D(M,g)$. We use the function $\tilde{\varphi_0}$ we defined before, namely
 \begin{align*} 
    \tilde{\varphi_0} : [0,L] \times \S^{n-1} & \longrightarrow \R \\ 
               (r, p) &  \longmapsto \left\{  
\begin{array}{cc}
     \varphi_0(r,p) & \mbox{ if } 0 \le r \le L/2  \\
     - \varphi_0(L-r,p) & \mbox{ if } L/2 \le r \le L. 
\end{array} \nonumber
\right.
\end{align*}
We already saw that $\tilde{\varphi_0}$ is continuous, that $\int_\Sigma \tilde{\varphi_0}dV_\Sigma =0$ and we can check that for all $i = 1, \ldots, m_1$,
\begin{align*}
    \int_\Sigma \tilde{\varphi_0}f_1^i dV_\Sigma =0.
\end{align*}
 Using $\tilde{\varphi_0}$ as a test function, we get 
 \begin{align*}
     \sigma_{m_1+1}(M,g) < \sigma_{(0)}^D(A_{1+L/2}),
 \end{align*}
 which is a sharp upper bound.
    \item Let $M=[0,L]\times \S^{n-1}$, with $L_1 \le L$. Then $\sigma_1(M, g) < \sigma_2(M, g) = \ldots = \sigma_{m_1+1}(M, g)$. We already dealt with this case in the proof of \Cref{thm : bound sigma2 sigmam1} and we saw that
    \begin{align*}
        \sigma_{m_1+1}(M, g) < \sigma_{(1)}^N(A_{1+L/2}),
    \end{align*}
 which is a sharp upper bound.

\end{enumerate}

Therefore, given a hypersurface of revolution $(M=[0, L] \times \S^{n-1}, g)$, we have a sharp upper bound for $\sigma_{m_1+1}(M,g)$, depending on $n$ and $L$, given by 
\begin{align*}
    \sigma_{m_1+1}(M,g) < B_n^{m_1+1}(L) := \left\{
    \begin{array}{ll}
         \sigma_{(2)}^N(A_{1+L/2}) & \mbox{ if } L \le L_2  \\
         \sigma_{(0)}^D(A_{1+L/2}) & \mbox{ if } L_2 < L \le L_1 \\
         \sigma_{(1)}^N(A_{1+L/2}) & \mbox{ if } L_1 < L.
    \end{array}
    \right.
\end{align*}

\begin{figure}[ht]
\centering
\includegraphics[scale=0.48]{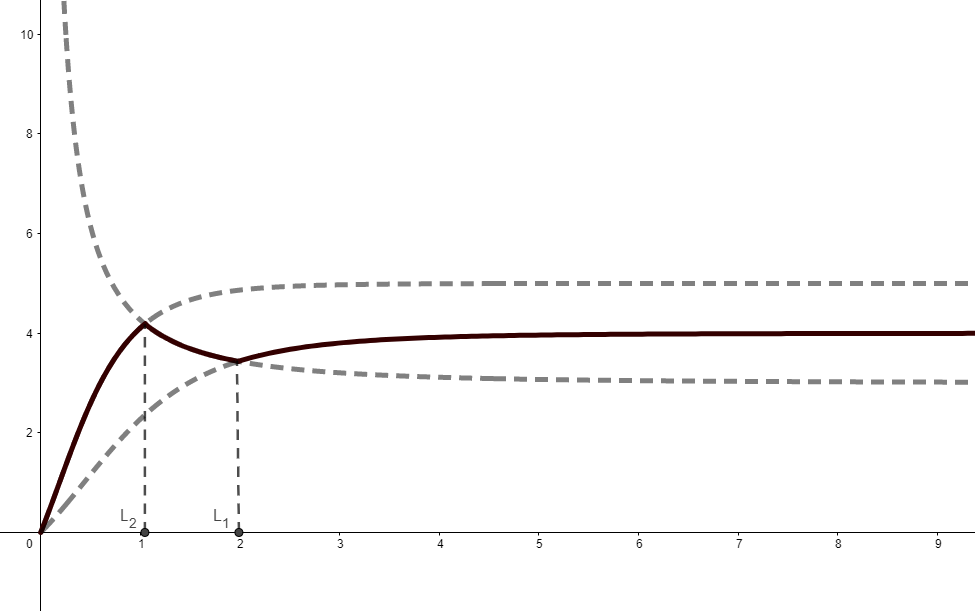}
\caption{Representation of the case $n = 5$. The solid curve is the bound given in \Cref{thm : bound sigmam1plus1}.
}
\label{fig : sigmam1plus1}
\end{figure}

The proof of sharpness goes as in the proof of \Cref{thm : principal deux}.

\end{proof}

From this, one can once again look for a sharp upper bound for $\sigma_{m_1+1}(M,g)$ that depends only on the dimension $n$ of $M$. This bound is given by
\begin{align} \label{eqn : cas m1plus1}
    \sigma_{m_1+1}(M,g) < B_n^{m_1+1} & := \max \left\{ \sigma_0^D(A_{1+{L_2}/2}), n-1\right\} \nonumber \\
    & =  \left\{
   \begin{array}{ll}
       \sigma_0^D(A_{1+{L_2}/2})  & \mbox{ if } 3 \le n \le 6  \\
        n-1 & \mbox{ if } 7 \le n.
    \end{array}
    \right.
\end{align}
A proof of (\ref{eqn : cas m1plus1}) is given in Appendix \ref{appen : cas m1plus1}.
\medskip

Therefore, the eigenvalue $k = m_1+1$ possesses a finite critical length if $3 \le n \le 6$, and it has a critical length at infinity if $7 \le n$.

\begin{rem} \label{rem : borne sans L}
It is then tempting to search for an expression for $B_n^k := \sup_{L \in \R_+^*} \{B_n^k(L)\}$ for any $n$ and $k$; but it seems to be hard to give an explicit formula for it. Indeed, as Sect. \ref{subsect : sigma2 sigmam1} and \ref{subsect : sigmam1plus1} suggest, the function $ L \longmapsto B_n^k(L)$ is hard to determine and can be either smooth (as in Sect.  \ref{subsect : sigma2 sigmam1} for instance) or piecewise smooth (as Sect.  \ref{subsect : sigmam1plus1} for instance). In the second case, there are possibly many irregular points that we have to consider. Moreover, depending on the value of $n$ and $k$: 
\begin{enumerate}
    \item Either $k$ has a finite critical length, i.e $B_n^k =B_n^k(L_k)$ for a certain $L_k \in \R_+^*$. That is for instance the case  of $\sigma_1(M,g)$ or $\sigma_{m_1+1}$ if $n=3, 4, 5 \mbox{ or } 6$;
    \item Or $k$ has a critical length at infinity, i.e $B_n^k = \lim_{L \to \infty} B_n^k(L)$. That is for instance the case of $\sigma_2(M,g), \ldots, \sigma_{m_1}(M,g)$.
\end{enumerate}
Furthermore, we will prove in Sect. \ref{sect : proof infinity de k avec L fini} that for all $n  \ge 3$, there are infinitely many $k$ that have a finite critical length associated to them. In all these cases, the function $ L \longmapsto B_n^k(L)$ is piecewise smooth.
\end{rem}

\section{Critical lengths of hypersurfaces of revolution} \label{sect : proof infinity de k avec L fini}

We recall that given $n \ge 3$, we are interested in giving information about the set of finite critical lengths. We want to prove \Cref{thm : principal 4}, i.e that there are infinitely many $k$ such that $B_n^k = B_n^k(L_k)$ for a certain finite $L_k \in \R_+^*$, and that the sequence of critical lengths converges to $0$.

\begin{proof}
As before, for $j \ge 0$, we denote by $m_j$ the number given by the formula (\ref{frm : multiplicite}), which is the multiplicity of $\sigma_{(j)}^D(A_R)$ as well as the multiplicity of $\sigma_{(j)}^N(A_R)$. 
Let $i \ge 2$ be an integer. We claim that for all $k$ such that 
\begin{align} \label{eqn : k donnant L etoile}
    m_0 + \sum_{j=1}^{i-1} 2m_j + m_i < k \le m_0 + \sum_{j=1}^{i} 2m_j,
\end{align}
we have 
\begin{align*}
    B_n^k = B_n^k(L_k)
\end{align*}
for a certain  $L_k \in \R_+^*$.
\medskip

Indeed, let $k$ satisfy (\ref{eqn : k donnant L etoile}). Then, because of the asymptotic behaviour of the functions $L \longmapsto \sigma_{(j)}^D(A_{1+L/2})$ and $L \longmapsto \sigma_{(j)}^N(A_{1+L/2})$ as $L \to \infty$,  there exists $C >0$ such that for all $L > C$, we have $B_n^k(L) = \sigma_{(i)}^D(A_{1+L/2})$. But we can compute that
\begin{align*}
    \frac{\partial}{\partial L} \sigma_{(i)}^D(A_{1+L/2}) = - \frac{4(L+2)(2i+n-2)^2(1+L/2)^{2i+n}}{\left(4(1+L/2)^{2i+n}-L^2-4L-4   \right)^2} < 0,
\end{align*}
which means that the function $L \longmapsto \sigma_{(i)}^D(A_{1+L/2})$ is strictly decreasing. Hence, for $L > L'  > C$, we have $B_n^k(L) < B_n^k(L')$. Therefore, for such a $k$, we have  
\begin{align*}
    B_n^k= B_n^k(L_k)
\end{align*}
with $L_k$ finite, that is $L_k$ is a  finite critical length associated to $k$.
\medskip

Then, defining $k_1 := 1$ and for each $i \ge 2$, defining $k_i := m_0 + \sum_{j=1}^{i} 2m_j$, we get a sequence $(k_i)_{i=1}^\infty$ such that 
\begin{align*}
    B_n^{k_i} = B_n^{k_i}(L_i)
\end{align*}
for a certain $L_i \in \R_+^*$ finite.
\medskip

Now we  want to  prove that the sequence of finite critical lengths $(L_i)_{i=1}^\infty$ converges to $0$.

Let $i \ge 1$. We know that $L_i$ has to be a solution of the equation $\sigma_{(j_1)}^D(A_{1+L/2}) = \sigma_{(j_2)}^N(A_{1+L/2})$ for a certain ordered pair $(j_1, j_2) \in \N^2$. As said before, for $L >C$, we have $B_n^{k_i}(L) = \sigma_{(i)}^D(A_{1+L/2})$, hence $L_i \le L_i^*$, where $L_i^*$ is the unique solution of the equation 
\begin{align*}
    \sigma_{(i)}^D(A_{1+L/2}) = \sigma_{(i+1)}^N(A_{1+L/2}).
\end{align*}
Therefore, in order to prove that $L_i \underset{i \to \infty}{\longrightarrow} 0$, we prove that $L_i^* \underset{i \to \infty}{\longrightarrow} 0$.

Using Propositions \ref{prop : SD} and \ref{prop : SN}, making some calculations and substituting $(1+L/2)$ by $R$, we can see that solving 
\begin{align*}
    \sigma_{(i)}^D(A_{1+L/2}) = \sigma_{(i+1)}^N(A_{1+L/2})
\end{align*}
is equivalent to finding the unique value $R_i \in (1, \infty)$ which solves the equation 
\begin{align*}
    (i+1)R^{2i+n-2}\left( R^{2i+n} - R^2(2i+n-1)- \frac{(i+n-1)(2i+n-1)}{i+1} \right) +(i+n-1)=0.
\end{align*}

We call 
\begin{align*}
   \underbrace{(i+1)R^{2i+n-2}    \left( \underbrace{ R^{2i+n} - R^2(2i+n-1)- \frac{(i+n-1)(2i+n-1)}{i+1} }_{=: \psi_i(R)} \right) +(i+n-1)}_{=: \Psi_i(R)}.
\end{align*}
Because $ (i+1)R^{2i+n-2} >0$ and $(i+n-1) >0$, then for $R_i$ to be the solution of the equation $\Psi_i(R) =0$, it is necessary that $\psi_i(R_i) <0$.
\medskip

Then we have 
\begin{align*}
    R^{2i +n} & < R^2(2i+n-1) + \frac{(i+n-1)(2i+n-1)}{i+1} \\
            & < R^2 \left( (2i+n-1) + \frac{(i+n-1)(2i+n-1)}{i+1}  \right) \\
            & = R^2 \left( \frac{(2i+n-1)(2i+n)}{i+1}  \right) \\
            & < R^2(2i+n)^2.
\end{align*}

Therefore we have 
\begin{align*}
    \ln(R) < \frac{2\ln(2i+n)}{2i+n-2}  
\end{align*}
and thus
\begin{align*}
    R < e^{\frac{2\ln(2i+n)}{2i+n-2} } 
\end{align*}

Remember that we substituted $(1+L/2)$ by $R$, and then  the unique solution of the equation 
\begin{align*}
    \sigma_{(i)}^D(A_{1+L/2}) = \sigma_{(i+1)}^N(A_{1+L/2})
\end{align*}
is a value $L_i^*$ wich satisfies
\begin{align*}
    0 < L_i^* < 2 \left( e^{\frac{2\ln(2i+n)}{2i+n-2}  }-1   \right). 
\end{align*}
Therefore, since $\frac{2 \ln(2i +n)}{2i +n -2} \underset{i \to \infty}{\longrightarrow} 0$, we have
\begin{align*}
    L_i < L_i^* \underset{i \to \infty}{\longrightarrow} 0. 
\end{align*}
In particular, for each $\delta >0$ there exists $k_0 \in \N$ such that for each $k >k_0$ which has a finite critical length $L_k$, then $L_k < \delta$.

\end{proof}

\begin{rem}
The condition given by (\ref{eqn : k donnant L etoile}) is sufficient, but is not a necessary one. Indeed, $k=1$ does not meet condition (\ref{eqn : k donnant L etoile}) but we have $B_n^1 = B_n^1(L_1)$, where $L_1$ is given by Corollary \ref{cor : corollaire du principal deux}. This consideration naturally leads to the following open question:
\end{rem}

\begin{question} \label{q : open}
Given $n \ge 3$, are there finitely or infinitely many $k \in \N$ such that $k$ has a critical length at infinity? 
\end{question}

\begin{appendices}

\section{Proof of equality (\ref{eqn : cas m1plus1})} \label{appen : cas m1plus1}

We know that there exists a unique $L_2 >0$ such that $\sigma_{(0)}^D(A_{1+L_2/2})= \sigma_{(2)}^N(A_{1+L_2/2})$. We want to choose, depending on the value of $n$, if $\sigma_{(0)}^D(A_{1+L_2/2})$ is bigger or smaller than $n-1$. For this purpose, we call $L_D$ the unique positive value such that $\sigma_{(0)}^D(A_{1+L_D/2})=n-1$, and we call $L_N$ the unique positive value such that $\sigma_{(2)}^N(A_{1+L_N/2}) = n-1$. 
\medskip

Then we have the following fact: if $L_N < L_D$, we have $\sigma_{(0)}(A_{1+L_2/2}) > n-1$. On the contrary, if $L_D < L_N$, we have $\sigma_{(0)}(A_{1+L_2/2}) < n-1$.
\begin{figure}[H]
    \centering
    \includegraphics[scale=0.5]{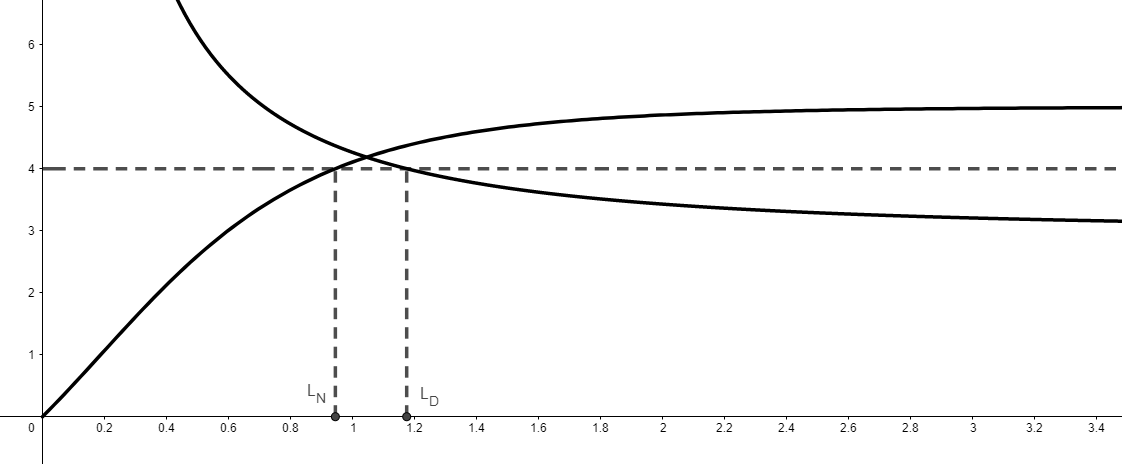}
    \caption{Representation of the case $n=5$. Since $L_N < L_D$, then $\sigma_{(0)}(A_{1+L_2/2}) > n-1$.}
    \label{fig: annexe}
\end{figure}

Hence, we solve the equation $\sigma_{(0)}^D(A_{1+L_D/2}) = n-1$, i.e we find the unique $L_D >0$ such that
\begin{align*}
    \frac{(n-2)(1+L_D/2)^{n-2}}{(1+L_D/2)^{n-2}-1}=n-1.
\end{align*}
We find 
\begin{align*}
    L_D = 2(n-1)^\frac{1}{n-2}-2.
\end{align*}
Similarly, solving the equation 
\begin{align*}
    \frac{2n((1+L_N/2)^{n+2}-1)}{2(1+L_N/2)^{n+2}+n} = n-1
\end{align*}
leads to 
\begin{align*}
    L_N = 2\left(\frac{n(n+1)}{2}\right)^{\frac{1}{n+2}}-2.
\end{align*}
We have to find for which values of $n$ we have $L_D < L_N$ and vice versa. This leads to the inequality
\begin{align*}
    \left( \frac{n(n+1)}{2} \right)^{\frac{1}{n+2}} > \left( n-1   \right)^{\frac{1}{n-2}},
\end{align*}
 which is equivalent to
 \begin{align*}
    \left( \frac{n(n+1)}{2} \right)^{\frac{n-2}{n+2}}  >  n-1 .
\end{align*}

We suppose $n \ge 9$.
\begin{align*}
    \left( \frac{n(n+1)}{2} \right)^{\frac{n-2}{n+2}} &> \left( \frac{n^2}{2} \right)^{\frac{n-2}{n+2}} = \frac{1}{2^{\frac{n-2}{n+2}}} n^{\frac{2n-4}{n+2}} 
     > \frac{1}{2}  n^{\frac{2n-4}{n+2}} 
   \ge \frac{1}{2} n^{\frac{14}{11}}.
\end{align*}
We analyze the function $f : [9, \infty ) \longrightarrow \R, x \longmapsto \frac{1}{2}x^{\frac{14}{11}}$. 
We have $f(9) > 8$, and $f'(x) = \frac{14}{22} x^{\frac{3}{11}}$. We can compute that $f'(9) >1$ and since $f''(x) = \frac{42}{242}x^{\frac{-9}{11}} >0$ for all $x \in [9, \infty)$, we can conclude $f'(x) >1 $ for all $x \in [9, \infty)$. Hence, $f(x) > x-1$ for all $x \in [9, \infty)$.
\medskip

Therefore, for all integers $n \ge 9$,  we have 
\begin{align*}
    \left( \frac{n(n+1)}{2} \right)^{\frac{n-2}{n+2}}  > n-1
\end{align*}
and then $L_D < L_N$. We can  compute  the cases $n=3, \ldots, 8$ and we can conclude that 
\begin{align*}
    \left\{
    \begin{array}{ll}
      L_N < L_D   & \mbox{ if } 3 \le n \le 6  \\
      L_D < L_N   & \mbox{ if } 7 \le n. 
    \end{array}
    \right.
\end{align*}
Therefore, we have
\begin{align*}
    \max \left\{ \sigma_0^D(A_{1+{L_2}/2}), n-1\right\} & = \left\{
    \begin{array}{ll}
       \sigma_0^D(A_{1+{L_2}/2})  & \mbox{ if } 3 \le n \le 6  \\
        n-1 & \mbox{ if } 7 \le n.
    \end{array}
    \right.
\end{align*}

\end{appendices}


\printbibliography

\end{document}